\documentclass[12pt, reqno]{amsart}
\usepackage{amsmath, amsthm, amscd, amsfonts, amssymb, graphicx, color}
\usepackage[bookmarksnumbered, colorlinks, plainpages]{hyperref}

\newtheorem{theorem}{Theorem}[section]

\theoremstyle{definition}

\theoremstyle{remark}
\newtheorem{remark}[theorem]{Remark}
\numberwithin{equation}{section}

\begin{document}
\setcounter{page}{1}


\title[On the two-dimensional moment problem]{On the two-dimensional moment problem}

\author[S. Zagorodnyuk]{Sergey Zagorodnyuk$^1$$^{*}$}







\date{Received: xxxxxx; Revised: yyyyyy; Accepted: zzzzzz.
\newline \indent $^{*}$ Corresponding author}

\begin{abstract}
In this paper we obtain an algorithm towards solving the two-dimensional moment problem.
This algorithm gives the necessary and sufficient conditions for the solvability of the moment problem.
It is shown that all solutions of the moment problem can be constructed using this algorithm.
In a consequence, analogous results are obtained for the complex moment problem.
\end{abstract} \maketitle

\section{Introduction and preliminaries}

In this paper we analyze the two-dimensional moment problem. Recall that this problem consists of
finding a non-negative Borel measure $\mu$ in $\mathbb{R}^2$
such that
\begin{equation}
\label{f1_1}
\int_{\mathbb{R}^2} x_1^m x_2^n d\mu = s_{m,n},\qquad m,n\in \mathbb{Z}_+,
\end{equation}
where $\{ s_{m,n} \}_{m,n\in \mathbb{Z}_+}$ is a prescribed sequence of complex numbers.

The two-dimensional moment problem and the (closely related to this subject) complex moment problem have
an extensive literature, see books~\cite{cit_10000_ST},
\cite{cit_20000_Akh},~\cite{cit_30000_B}, surveys~\cite{cit_40000_F},\cite{cit_50000_B} and~\cite{cit_60000_S}.
Some conditions of solvability for this moment problem were obtained by Kilpi and by Stochel and
Szafraniec, see e.g.~\cite{cit_20000_Akh} and \cite{cit_60000_S}. However, these conditions are hard to check.
Putinar and Vasilescu derived conditions of solvability and a description of all solutions
by means of a dimensional extension~\cite{cit_PuVa} (even for the $N$-dimensional moment problem).
The two-dimensional moment problem is solvable if and only if the prescribed sequence of moments can be
extended to a sequence $\{ s_{m,n,k} \}_{m,n,k\in \mathbb{Z}_+}$, satisfying some easy conditions
(including the positivity condition).
This extended sequence is the moment sequence for an extended moment problem:
\begin{equation}
\label{f1_1_1}
\int_{\mathbb{R}^2} x_1^m x_2^n \frac{1}{ \left( 1+x_1^2+x_2^2 \right)^k } d\mu = s_{m,n,k},\qquad m,n,k\in \mathbb{Z}_+.
\end{equation}
The unique solution of the moment problem~(\ref{f1_1_1}) provides a solution of the two-dimensional
moment problem. In this way all different extensions define all different solutions of the
two-dimensional moment problem.
However, it is not clear whether such extensions exist and what is
a procedure for the construction of extensions
$\{ s_{m,n,k} \}_{m,n,k\in \mathbb{Z}_+}$.

\noindent
The method of our investigation uses an abstract operator approach, see~\cite{cit_70000_Z}.
Firstly, we obtain a solvability criterion for an auxiliary extended two-dimensional moment problem.
This moment problem is somewhat similar to the moment problem~(\ref{f1_1_1}) but we do not see
any direct relationship.
It is shown that the extended two-dimensional moment problem is always
determinate and its solution can be constructed explicitly.

An idea of our algorithm is to extend the symmetric operators related to the two-dimensional moment problem,
not "entirely", but on a discrete set of points.
It is shown that all solutions of the moment problem~(\ref{f1_1}) can be constructed on this way.
Roughly speaking, the final algorithm reduces to the solving of finite and infinite linear systems of equations with
parameters.

\noindent
In a consequence, analogous results are obtained for the complex moment problem.

{\bf Notations. } As usual, we denote by $\mathbb{R},\mathbb{C},\mathbb{N},\mathbb{Z},\mathbb{Z}_+$
the sets of real numbers, complex numbers, positive integers, integers and non-negative integers,
respectively. The real plane will be denoted by $\mathbb{R}^2$.
For a subset $S$ of $\mathbb{R}^2$ we denote by $\mathfrak{B}(S)$ the set of all Borel subsets of $S$.
Everywhere in this paper, all Hilbert spaces are assumed to be separable. By
$(\cdot,\cdot)_H$ and $\| \cdot \|_H$ we denote the scalar product and the norm in a Hilbert space $H$,
respectively. The indices may be omitted in obvious cases.
For a set $M$ in $H$, by $\overline{M}$ we mean the closure of $M$ in the norm $\| \cdot \|_H$. For
$\{ x_{m,n} \}_{m,n\in \mathbb{Z}_+}$, $x_{m,n}\in H$, we write
$\mathop{\rm Lin}\nolimits \{ x_{m,n} \}_{m,n\in \mathbb{Z}_+}$ for the set of linear combinations of elements
$\{ x_{m,n} \}_{m,n\in \mathbb{Z}_+}$
and $\mathop{\rm span}\nolimits \{ x_{m,n} \}_{m,n\in \mathbb{Z}_+} =
\overline{ \mathop{\rm Lin}\nolimits \{ x_{m,n} \}_{m,n\in \mathbb{Z}_+} }$.
The identity operator in $H$ is denoted by $E_H$. For an arbitrary linear operator $A$ in $H$,
the operators $A^*$,$\overline{A}$,$A^{-1}$ mean its adjoint operator, its closure and its inverse
(if they exist). By $D(A)$ and $R(A)$ we mean the domain and the range of the operator $A$.
The norm of a bounded operator $A$ is denoted by $\| A \|$.
By $P^H_{H_1} = P_{H_1}$ we mean the operator of orthogonal projection in $H$ on a subspace
$H_1$ in $H$.
By $L^2_\mu$ we denote the usual space of square-integrable complex functions $f(x_1,x_2)$,
$x_1,x_2\in \mathbb{R}^2$, with respect to the Borel measure $\mu$ in $\mathbb{R}^2$.

\section{The solution of an extended two-dimensional moment problem.}
Consider the following moment problem: to find a non-negative Borel measure $\mu$ in $\mathbb{R}^2$
such that
$$ \int_{\mathbb{R}^2} x_1^m (x_1+i)^k (x_1-i)^l  x_2^n (x_2+i)^r (x_2-i)^t d\mu = u_{m,k,l;n,r,t},\qquad $$
\begin{equation}
\label{f2_1}
m,n\in \mathbb{Z}_+,\ k,l,r,t \in \mathbb{Z},
\end{equation}
where $\{ u_{m,k,l;n,r,t} \}_{m,n\in \mathbb{Z}_+,k,l,r,t \in \mathbb{Z}}$ is a prescribed
sequence of complex numbers.
This problem is said to be {\bf the extended two-dimensional moment problem}.

\noindent
We set
$$ \Omega = \{ (m,k,l;n,r,t):\ m,n\in \mathbb{Z}_+,\ k,l,r,t \in \mathbb{Z} \}, $$
$$ \Omega_0 = \{ (m,k,l;n,r,t):\ m,n\in \mathbb{Z}_+,\ k,l,r,t \in \mathbb{Z},\ k=l=r=t=0 \}, $$
$$ \Omega' = \Omega \backslash \Omega_0. $$
Let the moment problem~(\ref{f2_1}) have a solution $\mu$. Choose an arbitrary function
$$ P(x_1,x_2) = \sum_{(m,k,l;n,r,t)\in \Omega} \alpha_{m,k,l;n,r,t} x_1^m (x_1+i)^k (x_1-i)^l
x_2^n (x_2+i)^r (x_2-i)^t, $$
where all but finite number of complex coefficients $\alpha_{m,k,l;n,r,t}$ are zeros. Then
$$ 0 \leq \int_{\mathbb{R}^2} |P(x_1,x_2)|^2 d\mu =
\sum_{(m,k,l;n,r,t),(m',k',l';n',r',t')\in \Omega} \alpha_{m,k,l;n,r,t}
\overline{ \alpha_{m',k',l';n',r',t'} } $$
$$ * \int_{\mathbb{R}^2}
x_1^{m+m'} (x_1+i)^{k+l'} (x_1-i)^{l+k'}  x_2^{n+n'} (x_2+i)^{r+t'} (x_2-i)^{t+r'}
d\mu $$
$$ = \sum_{(m,k,l;n,r,t),(m',k',l';n',r',t')\in \Omega} \alpha_{m,k,l;n,r,t}
\overline{ \alpha_{m',k',l';n',r',t'} } u_{m+m',k+l',l+k';n+n',r+t',t+r'}. $$
Therefore
\begin{equation}
\label{f2_2}
\sum_{(m,k,l;n,r,t),(m',k',l';n',r',t')\in \Omega} \alpha_{m,k,l;n,r,t}
\overline{ \alpha_{m',k',l';n',r',t'} } u_{m+m',k+l',l+k';n+n',r+t',t+r'} \geq 0,
\end{equation}
for arbitrary complex coefficients $\alpha_{m,k,l;n,r,t}$, where
all but finite number of $\alpha_{m,k,l;n,r,t}$ are zeros. The latter condition on the coefficients
$\alpha_{m,k,l;n,r,t}$ in infinite sums will be assumed in similar situations.

We shall use the following important fact (e.g.~\cite[pp.361-363]{cit_80000_AG}).
\begin{theorem}
\label{t2_1}
Let a sequence of complex numbers $\{ u_{m,k,l;n,r,t} \}_{(m,k,l;n,r,t)\in \Omega}$
satisfy condition~(\ref{f2_2}).
Then there exist a separable Hilbert space $H$ with a scalar product $(\cdot,\cdot)_H$ and
a sequence $\{ x_{m,k,l;n,r,t} \}_{(m,k,l;n,r,t)\in \Omega}$ in $H$, such that
$$ ( x_{m,k,l;n,r,t}, x_{m',k',l';n',r',t'} )_H = u_{m+m',k+l',l+k';n+n',r+t',t+r'}, $$
\begin{equation}
\label{f2_3}
(m,k,l;n,r,t),(m',k',l';n',r',t')\in \Omega,
\end{equation}
and $\mathop{\rm span}\nolimits\{ x_{m,k,l;n,r,t} \}_{(m,k,l;n,r,t)\in \Omega} = H$.
\end{theorem}
\begin{proof}
(We do not claim the originality of the idea of this proof).
Choose an arbitrary infinite-dimensional linear vector space $V$ (for instance, one
may choose the space of all complex sequences $(u_n)_{n\in \mathbb{N}}$, $u_n\in \mathbb{C}$).
Let $X = \{ x_{m,k,l;n,r,t} \}_{(m,k,l;n,r,t)\in \Omega}$ be an
arbitrary infinite sequence of linear independent elements
in $V$ which is indexed by elements  of $\Omega$.
Set $L_X = \mathop{\rm Lin}\nolimits\{ x_{m,k,l;n,r,t} \}_{(m,k,l;n,r,t)\in \Omega}$.
Introduce the following functional:
$$ [x,y] = \sum_{(m,k,l;n,r,t),(m',k',l';n',r',t')\in \Omega} \alpha_{m,k,l;n,r,t}
\overline{\beta_{m',k',l';n',r',t'}} $$
\begin{equation}
\label{f2_4}
* u_{m+m',k+l',l+k';n+n',r+t',t+r'},
\end{equation}
for $x,y\in L_X$,
$$ x=\sum_{(m,k,l;n,r,t)\in\Omega} \alpha_{m,k,l;n,r,t} x_{m,k,l;n,r,t}, $$
$$ y=\sum_{(m',k',l';n',r',t')\in\Omega} \beta_{m',k',l';n',r',t'} x_{m',k',l';n',r',t'}, $$
where $\alpha_{m,k,l;n,r,t},\beta_{m',k',l';n',r',t'}\in \mathbb{C}$.
Here all but finite number of indices $\alpha_{m,k,l;n,r,t},\beta_{m',k',l';n',r',t'}$ are zeros.

\noindent
The set $L_X$ with $[\cdot,\cdot]$ will be a quasi-Hilbert space. Factorizing and making the completion
we obtain the  desired space $H$ (e.g. \cite{cit_30000_B}).
\end{proof}

Let the moment problem~(\ref{f2_1}) be given and the condition~(\ref{f2_2}) hold.
By Theorem~\ref{t2_1} there exist  a Hilbert space $H$ and a sequence
$\{ x_{m,k,l;n,r,t} \}_{(m,k,l;n,r,t)\in \Omega}$, in $H$, such that
relation~(\ref{f2_3}) holds.
Set $L = \mathop{\rm Lin}\nolimits\{ x_{m,k,l;n,r,t} \}_{(m,k,l;n,r,t)\in \Omega}$.
Introduce the following operators
\begin{equation}
\label{f2_5}
A_0 \sum_{(m,k,l;n,r,t)\in \Omega} \alpha_{m,k,l;n,r,t} x_{m,k,l;n,r,t} =
\sum_{(m,k,l;n,r,t)\in \Omega} \alpha_{m,k,l;n,r,t} x_{m+1,k,l;n,r,t},
\end{equation}
\begin{equation}
\label{f2_6}
B_0 \sum_{(m,k,l;n,r,t)\in \Omega} \alpha_{m,k,l;n,r,t} x_{m,k,l;n,r,t} =
\sum_{(m,k,l;n,r,t)\in \Omega} \alpha_{m,k,l;n,r,t} x_{m,k,l;n+1,r,t},
\end{equation}
where all but finite number of complex coefficients $\alpha_{m,k,l;n,r,t}$ are zeros.
Let us check that these definitions are correct.
Indeed, suppose that
\begin{equation}
\label{f2_7}
x = \sum_{(m,k,l;n,r,t)\in \Omega} \alpha_{m,k,l;n,r,t} x_{m,k,l;n,r,t} =
\sum_{(m',k',l';n',r',t')\in \Omega} \beta_{m',k',l';n',r',t'} x_{m',k',l';n',r',t'}.
\end{equation}
We may write
$$ \left( \sum_{(m,k,l;n,r,t)\in \Omega} \alpha_{m,k,l;n,r,t} x_{m+1,k,l;n,r,t}, x_{a,b,c;d,e,f} \right)_H $$
$$ = \sum_{(m,k,l;n,r,t)\in \Omega} \alpha_{m,k,l;n,r,t} u_{m+1+a,k+c,l+b;n+d,r+f,t+e} $$
$$= \sum_{(m,k,l;n,r,t)\in \Omega} \alpha_{m,k,l;n,r,t} (x_{m,k,l;n,r,t}, x_{a+1,b,c;d,e,f})_H $$
$$ = (x,x_{a+1,b,c;d,e,f})_H,\quad (a,b,c;d,e,f)\in \Omega. $$
In the same manner we obtain:
$$ \left( \sum_{(m',k',l';n',r',t')\in \Omega} \alpha_{m',k',l';n',r',t'} x_{m'+1,k',l';n',r',t'},
x_{a,b,c;d,e,f} \right)_H $$
$$ = (x,x_{a+1,b,c;d,e,f})_H,\quad (a,b,c;d,e,f)\in \Omega. $$
Therefore the  definition of $A_0$ is correct. The correctness of the definition of $B_0$ can be checked in
a similar manner.
Notice that
$$ (A_0 x_{m,k,l;n,r,t}, x_{a,b,c;d,e,f} )_H = ( x_{m+1,k,l;n,r,t}, x_{a,b,c;d,e,f} )_H $$
$$ = u_{m+1+a,k+c,l+b;n+d,r+f,t+e} = ( x_{m,k,l;n,r,t}, x_{a+1,b,c;d,e,f} )_H $$
$$ = ( x_{m,k,l;n,r,t}, A_0 x_{a,b,c;d,e,f} )_H,\quad (m,k,l;n,r,t),(a,b,c;d,e,f)\in\Omega. $$
Therefore $A_0$ is symmetric. The same argument implies that $B_0$ is symmetric, as well.

\noindent
Suppose that the following conditions hold:
$$ u_{m+1+a,k+c,l+b;n+d,r+f,t+e} + i u_{m+a,k+c,l+b;n+d,r+f,t+e} $$
\begin{equation}
\label{f2_8}
= u_{m+a,k+1+c,l+b;n+d,r+f,t+e},
\end{equation}
$$ u_{m+1+a,k+c,l+b;n+d,r+f,t+e} - i u_{m+a,k+c,l+b;n+d,r+f,t+e} $$
\begin{equation}
\label{f2_9}
= u_{m+a,k+c,l+1+b;n+d,r+f,t+e},
\end{equation}
$$ u_{m+a,k+c,l+b;n+1+d,r+f,t+e} + i u_{m+a,k+c,l+b;n+d,r+f,t+e} $$
\begin{equation}
\label{f2_10}
= u_{m+a,k+c,l+b;n+d,r+1+f,t+e},
\end{equation}
$$ u_{m+a,k+c,l+b;n+1+d,r+f,t+e} - i u_{m+a,k+c,l+b;n+d,r+f,t+e} $$
\begin{equation}
\label{f2_11}
= u_{m+a,k+c,l+b;n+d,r+f,t+1+e},
\end{equation}
for all $(m,k,l;n,r,t),(a,b,c;d,e,f)\in\Omega$.
These conditions are equivalent to conditions
$$ (x_{m+1,k,l;n,r,t} + i x_{m,k,l;n,r,t}, x_{a,b,c;d,e,f})_H  $$
\begin{equation}
\label{f2_12}
= (x_{m,k+1,l;n,r,t}, x_{a,b,c;d,e,f})_H,
\end{equation}
$$ (x_{m+1,k,l;n,r,t} - i x_{m,k,l;n,r,t}, x_{a,b,c;d,e,f})_H $$
\begin{equation}
\label{f2_13}
= (x_{m,k,l+1;n,r,t}, x_{a,b,c;d,e,f})_H,
\end{equation}
$$ (x_{m,k,l;n+1,r,t} + i x_{m,k,l;n,r,t}, x_{a,b,c;d,e,f})_H $$
\begin{equation}
\label{f2_14}
= (x_{m,k,l;n,r+1,t}, x_{a,b,c;d,e,f})_H,
\end{equation}
$$ (x_{m,k,l;n+1,r,t} - i x_{m,k,l;n,r,t}, x_{a,b,c;d,e,f})_H $$
\begin{equation}
\label{f2_15}
= (x_{m,k,l;n,r,t+1}, x_{a,b,c;d,e,f})_H,
\end{equation}
for all $(m,k,l;n,r,t),(a,b,c;d,e,f)\in\Omega$.
The latter conditions are equivalent to the following conditions:
\begin{equation}
\label{f2_16}
x_{m+1,k,l;n,r,t} + i x_{m,k,l;n,r,t} = x_{m,k+1,l;n,r,t},
\end{equation}
\begin{equation}
\label{f2_17}
x_{m+1,k,l;n,r,t} - i x_{m,k,l;n,r,t} = x_{m,k,l+1;n,r,t},
\end{equation}
\begin{equation}
\label{f2_18}
x_{m,k,l;n+1,r,t} + i x_{m,k,l;n,r,t} = x_{m,k,l;n,r+1,t},
\end{equation}
\begin{equation}
\label{f2_19}
x_{m,k,l;n+1,r,t} - i x_{m,k,l;n,r,t} = x_{m,k,l;n,r,t+1},
\end{equation}
for all $(m,k,l;n,r,t)\in\Omega$. The last conditions mean that
\begin{equation}
\label{f2_20}
(A_0 + i E_H) x_{m,k,l;n,r,t} = x_{m,k+1,l;n,r,t},
\end{equation}
\begin{equation}
\label{f2_21}
(A_0 - i E_H) x_{m,k,l;n,r,t} = x_{m,k,l+1;n,r,t},
\end{equation}
\begin{equation}
\label{f2_22}
(B_0 + i E_H) x_{m,k,l;n,r,t} = x_{m,k,l;n,r+1,t},
\end{equation}
\begin{equation}
\label{f2_23}
(B_0 - i E_H) x_{m,k,l;n,r,t} = x_{m,k,l;n,r,t+1},
\end{equation}
for all $(m,k,l;n,r,t)\in\Omega$.
The latter conditions imply that
$$ (A_0 \pm iE_H) L =  L,\quad (B_0 \pm iE_H) L =  L. $$
Therefore operators $A_0$ and $B_0$ are essentially self-adjoint.
The conditions~(\ref{f2_20})-(\ref{f2_23}) also imply that
\begin{equation}
\label{f2_24}
(A_0 + i E_H)^{-1} x_{m,k,l;n,r,t} = x_{m,k-1,l;n,r,t},
\end{equation}
\begin{equation}
\label{f2_25}
(A_0 - i E_H)^{-1} x_{m,k,l;n,r,t} = x_{m,k,l-1;n,r,t},
\end{equation}
\begin{equation}
\label{f2_26}
(B_0 + i E_H)^{-1} x_{m,k,l;n,r,t} = x_{m,k,l;n,r-1,t},
\end{equation}
\begin{equation}
\label{f2_27}
(B_0 - i E_H)^{-1} x_{m,k,l;n,r,t} = x_{m,k,l;n,r,t-1},
\end{equation}
for all $(m,k,l;n,r,t)\in\Omega$.

Consider the Cayley transformations of $A_0$ and $B_0$:
\begin{equation}
\label{f2_28}
V_{A_0} = (A_0-iE_H)(A_0 + i E_H)^{-1},
\end{equation}
\begin{equation}
\label{f2_29}
V_{B_0} = (B_0-iE_H)(B_0 + i E_H)^{-1},\qquad D(A_0)=D(B_0) = L.
\end{equation}
By virtue of relations~(\ref{f2_21}),(\ref{f2_23}),(\ref{f2_24}),(\ref{f2_26}) we obtain:
$$ V_{A_0} V_{B_0} x_{m,k,l;n,r,t} = x_{m,k-1,l+1;n,r-1,t+1} = V_{B_0} V_{A_0} x_{m,k,l;n,r,t}, $$
for all $(m,k,l;n,r,t)\in\Omega$. Therefore
\begin{equation}
\label{f2_30}
V_{A_0} V_{B_0} x = V_{B_0} V_{A_0} x,\qquad x\in L.
\end{equation}
By continuity we extend the isometric operators $V_{A_0}$ and $V_{B_0}$ to unitary operators
$U_{A_0}$ and $V_{B_0}$ in $H$, respectively. By continuity we conclude that
\begin{equation}
\label{f2_31}
U_{A_0} U_{B_0} x = U_{B_0} U_{A_0} x,\qquad x\in H.
\end{equation}
Set $A=\overline{A_0}$, $B=\overline{B_0}$. The Cayley transformations of the self-adjoint operrators
$A$ and $B$ coincide on
$L$ with $U_{A_0}$ and $U_{B_0}$, respectively. Thus, the Cayley transformations of $A$ and $B$
are $U_{A_0}$ and $U_{B_0}$, respectively. Therefore, operators $A$ and $B$ commute.

Notice that
\begin{equation}
\label{f2_32}
x_{m,k,l;n,r,t} = A^m (A+i)^k (A-i)^l B^n (B+i)^r (B-i)^t x_{0,0,0;0,0,0},
\end{equation}
for all $(m,k,l;n,r,t)\in\Omega$.
In fact, by induction we can check that
$$ x_{m,k,l;n,r,t} = (B-i E_H)^t x_{m,k,l;n,r,0},\qquad t\in \mathbb{Z}, $$
for any fixed $m,n\in \mathbb{Z}_+$, $k,l,r\in \mathbb{Z}$;
$$ x_{m,k,l;n,r,0} = (B+i E_H)^r x_{m,k,l;n,0,0},\qquad r\in \mathbb{Z}, $$
for any fixed $m,n\in \mathbb{Z}_+$, $k,l\in \mathbb{Z}$;
$$ x_{m,k,l;n,0,0} = B^n x_{m,k,l;0,0,0},\qquad n\in \mathbb{Z}_+, $$
for any fixed $m\in \mathbb{Z}_+$, $k,l\in \mathbb{Z}$;
$$ x_{m,k,l;0,0,0} = (A-i E_H)^l x_{m,k,0;0,0,0},\qquad l\in \mathbb{Z}, $$
for any fixed $m\in \mathbb{Z}_+$, $k\in \mathbb{Z}$;
$$ x_{m,k,0;0,0,0} = (A+i E_H)^l x_{m,0,0;0,0,0},\qquad k\in \mathbb{Z}, $$
for any fixed $m\in \mathbb{Z}_+$;
$$ x_{m,0,0;0,0,0} = A^m x_{0,0,0;0,0,0},\qquad m\in \mathbb{Z}_+, $$
and then by substitution of each relation into previous one we obtain relation~(\ref{f2_32}).

For the commuting self-adjoint operators $A$ and $B$ there exists an orthogonal operator spectral
measure $E(x)$ on $\mathfrak{B}(\mathbb{R}^2)$ such that
\begin{equation}
\label{f2_33}
A = \int_{\mathbb{R}^2} x_1 dE(x),\quad B = \int_{\mathbb{R}^2} x_2 dE(x).
\end{equation}
Then
$$ u_{m,k,l;n,r,t} = (x_{m,k,l;n,r,t}, x_{0,0,0;0,0,0})_H $$
$$ =
\left( \int_{\mathbb{R}^2} x_1^m (x_1+i)^k (x_1-i)^l x_2^n (x_2+i)^r (x_2-i)^t dE(x) x_{0,0,0;0,0,0},
x_{0,0,0;0,0,0}   \right)_H $$
$$ = \int_{\mathbb{R}^2} x_1^m (x_1+i)^k (x_1-i)^l x_2^n (x_2+i)^r (x_2-i)^t d (E(x) x_{0,0,0;0,0,0},
x_{0,0,0;0,0,0} )_H. $$
Hence, the Borel measure
\begin{equation}
\label{f2_34}
\mu = (E(x) x_{0,0,0;0,0,0}, x_{0,0,0;0,0,0} )_H,
\end{equation}
is a solution of the moment problem~(\ref{f2_1}).

\begin{theorem}
\label{t2_2}
Let the extended two-dimensional moment problem~(\ref{f2_1}) be given. The moment problem has a solution
if and only if conditions~(\ref{f2_2}) and~(\ref{f2_8})-(\ref{f2_11}) are satisfied.
If these conditions are satisfied then the solution of the moment problem is unique and can
be constructed by~(\ref{f2_34}).
\end{theorem}
\begin{proof}
The sufficiency of conditions~(\ref{f2_2}) and~(\ref{f2_8})-(\ref{f2_11}) for the existence of
a solution  of the moment problem~(\ref{f2_1}) was shown before the statement of the Theorem.
The necessity of condition~(\ref{f2_2}) was proved, as well. Let us check that
conditions~(\ref{f2_8})-(\ref{f2_11}) are necessary for the solvability of the moment problem~(\ref{f2_1}).

\noindent
Let $\mu$ be a solution of the moment problem~(\ref{f2_1}).
Consider the space $L^2_\mu$ and the following subsets in $L^2_\mu$:
\begin{equation}
\label{f2_35}
L_\mu = \mathop{\rm Lin}\nolimits \{ x_1^m (x_1+i)^k (x_1-i)^l x_2^n (x_2+i)^r
(x_2-i)^t \}_{(m,k,l;n,r,t)\in \Omega},\quad H_\mu = \overline{L_\mu}.
\end{equation}
We denote
\begin{equation}
\label{f2_36}
y_{m,k,l;n,r,t} :=  x_1^m (x_1+i)^k (x_1-i)^l x_2^n (x_2+i)^r (x_2-i)^t,\qquad (m,k,l;n,r,t)\in \Omega.
\end{equation}
Notice that
\begin{equation}
\label{f2_37}
(y_{m,k,l;n,r,t},y_{m',k',l';n',r',t'})_{L^2_\mu} = u_{m+m',k+l',l+k'; n+n',r+t',t+r'},
\end{equation}
for all $(m,k,l;n,r,t),(m',k',l';n',r',t')\in \Omega$.
Consider the operators of multiplication by the independent variable in $L^2_\mu$:
\begin{equation}
\label{f2_38}
A_{\mu} f(x_1,x_2) = x_1 f(x_1,x_2),\ B_{\mu} f(x_1,x_2) = x_2 f(x_1,x_2),\quad f\in L^2_\mu.
\end{equation}
Notice that
\begin{equation}
\label{f2_39}
(A_{\mu} + iE_{L^2_\mu}) y_{m,k,l;n,r,t} = y_{m,k+1,l;n,r,t},
\end{equation}
\begin{equation}
\label{f2_40}
(A_{\mu} - iE_{L^2_\mu}) y_{m,k,l;n,r,t} = y_{m,k,l+1;n,r,t},
\end{equation}
\begin{equation}
\label{f2_41}
(B_{\mu} + iE_{L^2_\mu}) y_{m,k,l;n,r,t} = y_{m,k,l;n,r+1,t},
\end{equation}
\begin{equation}
\label{f2_42}
(B_{\mu} - iE_{L^2_\mu}) y_{m,k,l;n,r,t} = y_{m,k,l;n,r,t+1},
\end{equation}
for all $(m,k,l;n,r,t),(m',k',l';n',r',t')\in \Omega$.

\noindent
Since conditions~(\ref{f2_2}) are satisfied, by Theorem~\ref{t2_1} there exist a Hilbert space $H$ and
a sequence of elements $\{ x_{m,k,l;n,r,t} \}_{(m,k,l;n,r,t)\in \Omega}$, in $H$, such that
relation~(\ref{f2_3}) holds. Repeating arguments after
the Proof of Theorem~\ref{t2_1} we construct operators $A_0$ and $B_0$ in $H$.
Consider the  following operator:
\begin{equation}
\label{f2_43}
W_0 \sum_{(m,k,l;n,r,t)\in \Omega} \alpha_{m,k,l;n,r,t} y_{m,k,l;n,r,t} =
\sum_{(m,k,l;n,r,t)\in \Omega} \alpha_{m,k,l;n,r,t} x_{m,k,l;n,r,t},
\end{equation}
where all but finite number of complex coefficients $\alpha_{m,k,l;n,r,t}$ are zeros.
Let us check that this operator is defined correctly.
In fact, suppose that
\begin{equation}
\label{f2_44}
x = \sum_{(m,k,l;n,r,t)\in \Omega} \alpha_{m,k,l;n,r,t} y_{m,k,l;n,r,t} =
\sum_{(m',k',l';n',r',t')\in \Omega} \beta_{m',k',l';n',r',t'} y_{m',k',l';n',r',t'},
\end{equation}
where $\beta_{m',k',l';n',r',t'}\in \mathbb{C}$.
We may write
$$ 0 = \left\| \sum_{(m,k,l;n,r,t)\in \Omega} (\alpha_{m,k,l;n,r,t} - \beta_{m,k,l;n,r,t}) y_{m,k,l;n,r,t}
\right\|_{L^2_\mu}^2 $$
$$ = \sum_{(m,k,l;n,r,t),(m',k',l';n',r',t')\in \Omega } (\alpha_{m,k,l;n,r,t} - \beta_{m,k,l;n,r,t})
 $$
$$ * \overline{ (\alpha_{m',k',l';n',r',t'} - \beta_{m',k',l';n',r',t'}) } (
y_{m,k,l;n,r,t}, y_{m',k',l';n',r',t'}
)_{L^2_\mu} $$
$$ = \sum_{(m,k,l;n,r,t),(m',k',l';n',r',t')\in \Omega } (\alpha_{m,k,l;n,r,t} - \beta_{m,k,l;n,r,t})
 $$
$$ * \overline{ (\alpha_{m',k',l';n',r',t'} - \beta_{m',k',l';n',r',t'}) } (
x_{m,k,l;n,r,t}, x_{m',k',l';n',r',t'}
)_H $$
$$ = \left\| \sum_{(m,k,l;n,r,t)\in \Omega} (\alpha_{m,k,l;n,r,t} - \beta_{m,k,l;n,r,t}) x_{m,k,l;n,r,t}
\right\|_H. $$
Thus, the operator $W_0$ is defined correctly.
If $\widetilde x\in H$ and
$$ \widetilde x = \sum_{(m,k,l;n,r,t)\in \Omega} \gamma_{m,k,l;n,r,t} y_{m,k,l;n,r,t}, $$
where $\gamma_{m,k,l;n,r,t}\in \mathbb{C}$,
then
$$ (W_0 x, W_0 \widetilde x)_H =
\sum_{(m,k,l;n,r,t),(m',k',l';n',r',t')\in \Omega } \alpha_{m,k,l;n,r,t}
\overline{ \gamma_{m',k',l';n',r',t'} } $$
$$ * ( x_{m,k,l;n,r,t}, x_{m',k',l';n',r',t'} )_H $$
$$ = \sum_{(m,k,l;n,r,t),(m',k',l';n',r',t')\in \Omega } \alpha_{m,k,l;n,r,t}
\overline{ \gamma_{m',k',l';n',r',t'} }
( y_{m,k,l;n,r,t}, y_{m',k',l';n',r',t'} )_H $$
$$ = (x,\widetilde x)_{L^2_\mu}. $$
By continuity we extend $W_0$ to a unitary operator $W$ which maps $H_\mu$ onto $H$.
Observe that
\begin{equation}
\label{f2_45}
W^{-1} A_0 W y_{m,k,l;n,r,t} = y_{m+1,k,l;n,r,t} = A_\mu y_{m,k,l;n,r,t},
\end{equation}
\begin{equation}
\label{f2_46}
W^{-1} B_0 W y_{m,k,l;n,r,t} = y_{m,k,l;n+1,r,t} = B_\mu y_{m,k,l;n,r,t},
\end{equation}
for all $(m,k,l;n,r,t)\in \Omega$.
By using the last relations in relations~(\ref{f2_39})-(\ref{f2_42}) we obtain
relations (\ref{f2_20})-(\ref{f2_23}). The latter relations are equivalent to
conditions~(\ref{f2_8})-(\ref{f2_11}).

Let us check that the solution of the moment problem is unique.
Consider the following transformation
$$ T:\ (x_1,x_2)\in \mathbb{R}^2 \mapsto (\varphi,\psi)\in [0,2\pi)\times[0,2\pi), $$
\begin{equation}
\label{f2_47}
e^{i\varphi} = \frac{x_1+i}{x_1-i},\quad e^{i\psi} = \frac{x_2+i}{x_2-i};
\end{equation}
and set
\begin{equation}
\label{f2_48}
\nu (\Delta) = \mu( T^{-1} (\Delta) ),\quad \Delta\in \mathfrak{B}([0,2\pi)\times[0,2\pi)).
\end{equation}
Since $T$ is a bijective continuous transformation, then $\nu$ is a non-negative Borel measure
on $[0,2\pi)\times[0,2\pi)$. Moreover, we have
\begin{equation}
\label{f2_49}
u_{0,k,-k;0,l,-l} = \int_{\mathbb{R}^2} \left( \frac{x_1+i}{x_1-i} \right)^k
\left( \frac{x_2+i}{x_2-i} \right)^l d\mu =
\int_{[0,2\pi)\times[0,2\pi)} e^{ik\varphi} e^{il\psi} d\nu,\quad
\end{equation}
for all $k,l\in \mathbb{Z}$.
Let $\widetilde\mu$ be another solution of the moment problem~(\ref{f2_1}) and
$\widetilde\nu$ be defined by
\begin{equation}
\label{f2_50}
\widetilde\nu (\Delta) = \widetilde\mu( T^{-1} (\Delta) ),\quad \Delta\in \mathfrak{B}([0,2\pi)\times[0,2\pi)).
\end{equation}
By relation~(\ref{f2_49}) we obtain that
\begin{equation}
\label{f2_51}
\int_{[0,2\pi)\times[0,2\pi)} e^{ik\varphi} e^{il\psi} d\nu =
\int_{[0,2\pi)\times[0,2\pi)} e^{ik\varphi} e^{il\psi} d\widetilde\nu,\quad k,l\in \mathbb{Z}.
\end{equation}
By the Weierstrass theorem we can approximate $\varphi^m$ and $\psi^n$, for some
fixed $m,n\in \mathbb{Z}_+$,
by trigonometric polynomials $P_k(\varphi)$ and $R_k(\psi)$, respectively:
\begin{equation}
\label{f2_51_1}
\max_{ \varphi\in [0,2\pi) } | \varphi^m - P_k(\varphi) |\leq \frac{1}{k},\quad
\max_{ \psi\in [0,2\pi) } | \psi^m - R_k(\psi) |\leq \frac{1}{k},\quad k\in \mathbb{N}.
\end{equation}
Then
$$ \left|
\int_{[0,2\pi)\times[0,2\pi)} \varphi^m \psi^n d\nu -
\int_{[0,2\pi)\times[0,2\pi)} P_k(\varphi) R_k(\psi) d\nu
\right| $$
$$
=
\left|
\int_{[0,2\pi)\times[0,2\pi)} (\varphi^m -P_k(\varphi))\psi^n d\nu +
\int_{[0,2\pi)\times[0,2\pi)} P_k(\varphi) (\psi^n - R_k(\psi)) d\nu
\right| $$
$$ \leq
\max_{ \psi\in [0,2\pi) } | \psi^m | \frac{1}{k} \nu([0,2\pi)) +
\max_{ \varphi\in [0,2\pi) } | P_k(\varphi) | \frac{1}{k} \nu([0,2\pi)) $$
$$ \leq
\max_{ \psi\in [0,2\pi) } | \psi^m | \frac{1}{k} \nu([0,2\pi)) +
\left( \frac{1}{k} + \max_{ \varphi\in [0,2\pi) } | \varphi^m | \right)
\frac{1}{k} \nu([0,2\pi))
\rightarrow 0, $$
as $k\to\infty$.
In the same manner we get
$$ \left|
\int_{[0,2\pi)\times[0,2\pi)} \varphi^m \psi^n d\widetilde\nu -
\int_{[0,2\pi)\times[0,2\pi)} P_k(\varphi) R_k(\psi) d\widetilde\nu
\right| \rightarrow 0, $$
as $k\to\infty$.
Hence, we  conclude that
\begin{equation}
\label{f2_52}
\int_{[0,2\pi)\times[0,2\pi)} \varphi^m \psi^n d\nu =
\int_{[0,2\pi)\times[0,2\pi)} \varphi^m \psi^n d\widetilde\nu,\quad m,n\in \mathbb{Z}_+.
\end{equation}
Since the two-dimensional moment problem on a rectangular has a unique solution, we get
$\nu = \widetilde\nu$ and $\mu = \widetilde\mu$.
\end{proof}

\section{An algorithm towards solving the two-dimensional moment problem.}
As a first application of our results on the extended two-dimensional moment problem we get the
following theorem.
\begin{theorem}
\label{t3_1}
Let the two-dimensional moment problem~(\ref{f1_1}) be given. The moment problem has a solution
if and only if there exists a sequence of complex numbers
$\{ u_{m,k,l;n,r,t} \}_{(m,k,l;n,r,t)\in\Omega}$, which
satisfies conditions~(\ref{f2_2}),~(\ref{f2_8})-(\ref{f2_11}) and
\begin{equation}
\label{f3_1}
u_{m,0,0;n,0,0} = s_{m,n},\qquad m,n\in \mathbb{Z}_+.
\end{equation}
\end{theorem}
The proof is obvious and left to the reader.

Let the two-dimensional moment problem~(\ref{f1_1}) be given.
As it is well known (and can be checked in the same manner as for the relation~(\ref{f2_2}))
the necessary condition for its solvability is the following:
\begin{equation}
\label{f3_2}
\sum_{ m,n,m',n' \in \mathbb{Z}_+ } \alpha_{m,n}
\overline{ \alpha_{m',n'} } s_{m+m',n+n'} \geq 0,
\end{equation}
for arbitrary complex coefficients $\alpha_{m,n}$, where
all but finite number of $\alpha_{m,n}$ are zeros.

\noindent
We assume that the condition~(\ref{f3_2}) holds.
Repeating arguments of the proof of Theorem~\ref{t2_1} we can state that there exist
a Hilbert space $\mathcal{H}_0$ and a sequence $\{ h_{m,n} \}_{m,n\in \mathbb{Z}_+}$ such that
\begin{equation}
\label{f3_3}
(h_{m,n},h_{m',n'})_{\mathcal{H}_0} = s_{m+m',n+n'},\quad m,n,m',n'\in \mathbb{Z}_+.
\end{equation}
Consider the following Hilbert space:
\begin{equation}
\label{f3_4}
\mathcal{H} = \mathcal{H}_0 \oplus \left(
\bigoplus_{j=1}^\infty \mathcal{H}_j
\right),
\end{equation}
where $\mathcal{H}_j$ are arbitrary one-dimensional Hilbert spaces, $j\in \mathbb{N}$.
We shall call it {\bf the model space for the two-dimensional moment problem}.

\noindent
Introduce an arbitrary indexation in the set $\Omega'$ by the unique positive integer index $j$:
\begin{equation}
\label{f3_5}
j\in \mathbb{N} \mapsto w = w(j)= (m,k,l;n,r,t)(j) \in \Omega'.
\end{equation}

Suppose that the two-dimensional moment problem~(\ref{f1_1}) has a solution $\mu$.
Consider the space $L^2_\mu$ and the following subsets in $L^2_\mu$:
\begin{equation}
\label{f3_6}
L_{\mu,0} = \mathop{\rm Lin}\nolimits \{ x_1^m x_2^n \}_{m,n\in \mathbb{Z}_+},\quad
H_{\mu,0} = \overline{L_{\mu,0}}.
\end{equation}
We denote
\begin{equation}
\label{f3_7}
y_{m,n} :=  x_1^m x_2^n,\qquad m,n\in \mathbb{Z}_+.
\end{equation}
Notice that
\begin{equation}
\label{f3_8}
(y_{m,n},y_{m',n'})_{L^2_\mu} = s_{m+m',n+n'},
\end{equation}
for all $m,n,m',n'\in \mathbb{Z}_+$.
We shall also use the notations from~(\ref{f2_35}),(\ref{f2_36}).

\noindent
Define the following numbers
$$ u_{m,k,l;n,r,t} :=
\int_{\mathbb{R}^2} x_1^m (x_1+i)^k (x_1-i)^l  x_2^n (x_2+i)^r (x_2-i)^t d\mu,\qquad $$
\begin{equation}
\label{f3_9}
(m,k,l;n,r,t)\in \Omega.
\end{equation}
For these numbers conditions~(\ref{f2_2}) hold and repeating arguments after the relation~(\ref{f2_42})
we construct a Hilbert space $H$ and
a sequence of elements $\{ x_{m,k,l;n,r,t} \}_{(m,k,l;n,r,t)\in \Omega}$, in $H$, such that
relation~(\ref{f2_3}) holds.
We introduce the operator $W$ as after~(\ref{f2_43}). The operator $W$ maps $H_\mu$ onto $H$.
Set
\begin{equation}
\label{f3_10}
H_0 = \mathop{\rm span}\nolimits \{ x_{m,0,0;n,0,0} \}_{m,n\in \mathbb{Z}_+} \subseteq H.
\end{equation}
Let us construct a sequence of Hilbert spaces $H_j$, $j\in \mathbb{N}$, in the following
way.

\noindent
{\bf Step 1.} We set
\begin{equation}
\label{f3_11}
f_1 = x_{ w(1) } - P^H_{H_0} x_{ w(1) },\quad H_1 = \mathop{\rm span}\nolimits \{ f_1 \},
\end{equation}
where $w(\cdot)$ is the indexation in the  set $\Omega'$.

\noindent
{\bf Step $r$, with $r\geq 2$}. We set
\begin{equation}
\label{f3_12}
f_r = x_{ w(r) } - P^H_{ H_0 \oplus ( \oplus_{1\leq t\leq r-1} H_t ) } x_{ w(r) },\quad
H_r = \mathop{\rm span}\nolimits \{ f_r \}.
\end{equation}
Then we get a representation
\begin{equation}
\label{f3_13}
H = H_0 \oplus \left(
\bigoplus_{j=1}^\infty H_j
\right).
\end{equation}
Observe that $H_j$ is either a one-dimensional Hilbert space or $H_j = \{ 0 \}$.
We denote
\begin{equation}
\label{f3_14}
\Lambda_\mu = \{ j\in \mathbb{N}:\ H_j \not= \{ 0 \} \},\qquad \Lambda_\mu' = \mathbb{N}\backslash \Lambda_\mu.
\end{equation}
Then
\begin{equation}
\label{f3_15}
H = H_0 \oplus \left(
\bigoplus_{j\in \Lambda_\mu} H_j
\right).
\end{equation}
We shall construct a unitary operator $U$ which maps $H$ onto the following subspace of
the model space $\mathcal{H}$:
\begin{equation}
\label{f3_16}
\widehat{\mathcal{H}} = \mathcal{H}_0 \oplus \left(
\bigoplus_{j\in \Lambda_\mu} \mathcal{H}_j
\right) \subset \mathcal{H}.
\end{equation}
Choose an arbitrary element
\begin{equation}
\label{f3_17}
x = \sum_{m,n\in \mathbb{Z}_+} \alpha_{m,n} x_{m,0,0;n,0,0} + \sum_{j\in\Lambda_\mu} \beta_j \frac{f_j}
{ \| f_j \|_H },
\end{equation}
with $\alpha_{m,n},\beta_j\in \mathbb{C}$.
Set
\begin{equation}
\label{f3_18}
Ux = \sum_{m,n\in \mathbb{Z}_+} \alpha_{m,n} h_{m,n} + \sum_{j\in\Lambda_\mu} \beta_j e_j,
\end{equation}
where $e_j\in \mathcal{H}_j$, $\| e_j \|_{\mathcal{H}} = 1$, are chosen arbitrarily.

\noindent
Let us check that this definition is correct. Suppose that $x$ has another representation:
\begin{equation}
\label{f3_19}
x = \sum_{m,n\in \mathbb{Z}_+} \widetilde\alpha_{m,n} x_{m,0,0;n,0,0} + \sum_{j\in\Lambda_\mu}
\widetilde\beta_j \frac{f_j}{ \| f_j \|_H },
\end{equation}
with $\widetilde\alpha_{m,n},\widetilde\beta_j\in \mathbb{C}$.
By orthogonality we have $\beta_j = \widetilde\beta_j$, $j\in \mathbb{N}$.
Then
$$ 0 = \left\| \sum_{m,n\in \mathbb{Z}_+} (\alpha_{m,n} - \widetilde\alpha_{m,n}) x_{m,0,0;n,0,0}
\right\|_{H}^2 $$
$$ = \sum_{m,n,m',n'\in \mathbb{Z}_+} (\alpha_{m,n} - \widetilde\alpha_{m,n})
 \overline{ (\alpha_{m',n'} - \widetilde\alpha_{m',n'}) } (
x_{m,0,0;n,0,0}, x_{m',0,0;n',0,0}
)_{H} $$
$$ = \sum_{m,n,m',n'\in \mathbb{Z}_+} (\alpha_{m,n} - \widetilde\alpha_{m,n})
 \overline{ (\alpha_{m',n'} - \widetilde\alpha_{m',n'}) } (
h_{m,n}, h_{m',n'}
)_{\mathcal{H}} $$
$$ = \left\| \sum_{m,n\in \mathbb{Z}_+} (\alpha_{m,n} - \widetilde\alpha_{m,n}) h_{m,n}
\right\|_\mathcal{H}. $$
Thus, the operator $U$ is defined correctly.
If $\widehat x\in H$ and
$$ \widehat x =
\sum_{m,n\in \mathbb{Z}_+} \widehat\alpha_{m,n} x_{m,0,0;n,0,0} + \sum_{j\in\Lambda_\mu}
\widehat\beta_j \frac{f_j}{ \| f_j \|_H },
$$
where $\widehat\alpha_{m,n}, \widehat\beta_j\in \mathbb{C}$,
then
$$ (x, \widehat x)_H =
\sum_{m,n,m',n'\in \mathbb{Z}_+} \alpha_{m,n}\overline{ \widehat\alpha_{m',n'} }
( x_{m,0,0;n,0,0}, x_{m',0,0;n',0,0} )_H + \sum_{j\in\Lambda_\mu} \beta_j \overline{ \widehat\beta_j }$$
$$ = \sum_{m,n,m',n'\in \mathbb{Z}_+} \alpha_{m,n}\overline{ \widehat\alpha_{m',n'} }
( h_{m,n}, h_{m',n'} )_\mathcal{H} + \sum_{j\in\Lambda_\mu} \beta_j \overline{ \widehat\beta_j } $$
$$ = (Ux, U\widehat x)_{\mathcal{H}}. $$
By continuity we extend $U$ to a unitary operator which maps $H$ onto $\widehat{\mathcal{H}}$.
Then the operator $UW$ is a unitary operator which maps $H_\mu$ onto $\widehat{\mathcal{H}}$.
We could define this operator directly, but we prefer to underline an abstract structure of the
corresponding spaces and this maybe explains where the model space comes from.

\noindent
We set
\begin{equation}
\label{f3_20}
h_{m,k,l;n,r,t} := UW y_{m,k,l;n,r,t},\quad
(m,k,l;n,r,t)\in\Omega.
\end{equation}
Observe that
\begin{equation}
\label{f3_21}
h_{m,0,0;n,0,0} = h_{m,n},\quad m,n\in \mathbb{Z}_+;\quad UW H_{\mu,0} = \mathcal{H}_0.
\end{equation}
Since
\begin{equation}
\label{f3_22}
x_{w(r)} \in H_0 \oplus \left(
\bigoplus_{j\in \Lambda_\mu:\ j\leq r} H_j
\right),
\end{equation}
then
\begin{equation}
\label{f3_23}
h_{w(r)} \in \mathcal{H}_0 \oplus \left(
\bigoplus_{j\in \Lambda_\mu:\ j\leq r} \mathcal{H}_j
\right),\qquad r\in \mathbb{N}.
\end{equation}
Observe that $\{ y_{m,k,l;n,r,t} \}_{ (m,k,l;n,r,t)\in\Omega }$ satisfy
relations~(\ref{f2_16})-(\ref{f2_19}) (with $y$ instead of $x$).
Therefore $\{ h_{m,k,l;n,r,t} \}_{ (m,k,l;n,r,t)\in\Omega }$ satisfy
relations, as well.
Notice that
$$ (h_{m,k,l;n,r,t},h_{m',k',l';n',r',t'})_{\mathcal{H}} =
(y_{m,k,l;n,r,t},y_{m',k',l';n',r',t'})_{L^2_\mu} $$
\begin{equation}
\label{f3_24}
 = u_{m+m',k+l',l+k';n+n',r+t',t+r'},
\end{equation}
for all $(m,k,l;n,r,t),(m',k',l';n',r',t')\in \Omega$.

\noindent
Choose an arbitrary $r\in \mathbb{N}$.
By~(\ref{f3_12}) we may write
$$ x_{ w(r) } = f_r + P^H_{ H_0 \oplus ( \oplus_{1\leq t\leq r-1} H_t ) } x_{ w(r) } $$
\begin{equation}
\label{f3_24_1}
= \| f_r \| \frac{f_r}{ \| f_r \| } + \sum_{t\in\Lambda_\mu:\ 1\leq t\leq r-1} \beta_t
\frac{f_t}{ \| f_t \| } + u_0,
\end{equation}
where $\beta_t\in \mathbb{C}$, $u_0\in H_0$.
By~(\ref{f3_20}) and~(\ref{f3_18}) we get
\begin{equation}
\label{f3_24_2}
h_{ w(r) } = Ux_{ w(r) } = \| f_r \| e_r + \sum_{t\in\Lambda_\mu:\ 1\leq t\leq r-1} \beta_t e_t
+ w_0,
\end{equation}
where $w_0 = U u_0 \in \mathcal{H}_0$. Therefore
\begin{equation}
\label{f3_24_3}
(h_{ w(r) }, e_r) \geq 0;
\end{equation}
and
\begin{equation}
\label{f3_24_4}
h_{w(r)} \in \mathcal{H}_0 \oplus \left(
\bigoplus_{t\in \Lambda_\mu:\ 1\leq t\leq r-1} \mathcal{H}_t
\right) \Leftrightarrow
(h_{ w(r) }, e_r) = 0 \Leftrightarrow f_r = 0 \Leftrightarrow r\in\Lambda_\mu'.
\end{equation}
In particular, we may write
\begin{equation}
\label{f3_24_5}
r\in\Lambda_\mu \Leftrightarrow
(h_{ w(r) }, e_r) > 0,\qquad r\in \mathbb{N}.
\end{equation}
\begin{theorem}
\label{t3_2}
Let the two-dimensional moment problem~(\ref{f1_1}) be given and condition~(\ref{f3_2}) holds.
Choose an arbitrary model
space $\mathcal{H}$ with a sequence $\{ h_{m,n} \}_{m,n\in \mathbb{Z}_+}$,
satisfying~(\ref{f3_3}) and fix it. The moment problem has a solution
if and only if there exists a sequence
$\{ h_{m,k,l;n,r,t} \}_{(m,k,l;n,r,t)\in\Omega}$, in $\mathcal{H}$ such that
the following conditions hold:

\begin{itemize}
\item[1)] $h_{m,0,0;n,0,0} = h_{m,n}$, $m,n\in \mathbb{Z}_+$;

\item[2)] 
$h_{w(r)} \in \mathcal{H}_0 \oplus \left(
\bigoplus_{j\in \Lambda:\ 0\leq j\leq r} \mathcal{H}_j
\right)$, and $(h_{w(r)},e_r) \geq 0$, $r\in \mathbb{N}$, for some subset $\Lambda\subset \mathbb{N}$.

\item[3)] The sequence $\{ h_{m,k,l;n,r,t} \}_{(m,k,l;n,r,t)\in\Omega}$
satisfies conditions~(\ref{f2_16})-(\ref{f2_19}) (with $h$ instead of $x$).

\item[4)] There exists a complex function $\varphi(m,k,l;n,r,t)$, $(m,k,l;n,r,t)\in\Omega$,
such that
$$ (h_{m,k,l;n,r,t},h_{m',k',l';n',r',t'})_{\mathcal{H}} =
\varphi(m+m',k+l',l+k';n+n',r+t',t+r'), $$
for all $(m,k,l;n,r,t),(m',k',l';n',r',t')\in \Omega$.
\end{itemize}
\end{theorem}
\begin{proof}
 The necessity of conditions~1)-4) for the solvability of the two-dimensional moment
problem was established before the statement of the Theorem.

\noindent
Let conditions~1),3),4) be satisfied. Consider the extended two-dimensional moment problem~(\ref{f2_1})
with
\begin{equation}
\label{f3_25}
u_{m,k,l;n,r,t} := \varphi(m,k,l;n,r,t),\qquad (m,k,l;n,r,t)\in \Omega,
\end{equation}
where $\varphi$ is from the condition~4).
Then
$$ \sum_{(m,k,l;n,r,t),(m',k',l';n',r',t')\in \Omega} \alpha_{m,k,l;n,r,t}
\overline{ \alpha_{m',k',l';n',r',t'} } u_{m+m',k+l',l+k';n+n',r+t',t+r'} $$
$$ = \sum_{(m,k,l;n,r,t),(m',k',l';n',r',t')\in \Omega} \alpha_{m,k,l;n,r,t}
\overline{ \alpha_{m',k',l';n',r',t'} } (h_{m,k,l;n,r,t}, h_{m',k',l';n',r',t'})_{\mathcal{H}} $$
$$ = \left\| \sum_{(m,k,l;n,r,t)\in \Omega} \alpha_{m,k,l;n,r,t} h_{m,k,l;n,r,t} \right\|_{\mathcal{H}}^2
\geq 0, $$
for arbitrary complex coefficients $\alpha_{m,k,l;n,r,t}$, where
all but finite number of $\alpha_{m,k,l;n,r,t}$ are zeros.

\noindent
By conditions~3) and~4) we conclude that conditions~(\ref{f2_8})-(\ref{f2_11}) hold.
By Theorem~\ref{t2_2} we obtain that there exists a non-negative Borel measure $\mu$ in
$\mathbb{R}^2$ such that~(\ref{f2_1}) holds. In particular, using conditions~4),1) we get
$$\int_{\mathbb{R}^2} x_1^m x_2^n d\mu = u_{m,0,0;n,0,0} = \varphi(m,0,0;n,0,0) $$
$$ = (h_{m,0,0;n,0,0},h_{0,0,0;0,0,0})_{\mathcal{H}} = (h_{m,n},h_{0,0})_{\mathcal{H}}
= s_{m,n},\quad m,n\in \mathbb{Z}_+. $$
\end{proof}

Observe that condition~2) can be removed from the statement of Theorem~\ref{t3_2}.
However, it will be used later.

Denote a set of sequences $\{ h_{m,k,l;n,r,t} \}_{(m,k,l;n,r,t)\in\Omega}$, in $\mathcal{H}$ satisfying
conditions 1)-4) by $X=X(\mathcal{H})$.
As we have seen in the proof of Theorem~\ref{t3_2},
for an arbitrary $\{ h_{m,k,l;n,r,t} \}_{(m,k,l;n,r,t)\in\Omega} \in X(\mathcal{H})$,
the unique solution of the extended two-dimensional moment problem with moments~(\ref{f3_25})
gives a solution of the two-dimensional moment problem.
Observe that {\it all solutions of the two-dimensional moment problem can be constructed in
this manner}.
Indeed, let $\mu$ be an arbitrary solution of the two-dimensional moment problem.
Repeating arguments from relation~(\ref{f3_6}) till the statement of Theorem~\ref{t3_2}
we may write
$$ \int_{\mathbb{R}_2} x_1^m (x_1+i)^k (x_1-i)^l  x_2^n (x_2+i)^r (x_2-i)^t d\mu =
(y_{m,k,l;n,r,t},y_{0,0,0;0,0,0})_{L^2_\mu} $$
$$ = (UWy_{m,k,l;n,r,t},UWy_{0,0,0;0,0,0})_{\mathcal{H}} =
(h_{m,k,l;n,r,t},h_{0,0,0;0,0,0})_{\mathcal{H}} $$
$$ = \varphi(m,k,l;n,r,t) = u_{m,k,l;n,r,t}, $$
for all $(m,k,l;n,r,t)\in\Omega$. Here the operators $U$,$W$,
the sequence $\{ h_{m,k,l;n,r,t} \}_{(m,k,l;n,r,t)\in\Omega}$, the function $\varphi(m,k,l;n,r,t)$
and the moments $u_{m,k,l;n,r,t}$, of course, depend on the choice of $\mu$.
Notice that $\{ h_{m,k,l;n,r,t} \}_{(m,k,l;n,r,t)\in\Omega} \in X(\mathcal{H})$.

\noindent
For such constructed parameters, the measure $\mu$ is a solution of the extended two-dimensional moment problem
considered in the proof of Theorem~\ref{t3_2}. Since the solution  of this moment problem is unique,
$\mu$ will be reconstructed in the above described manner.

Notice that condition~4) of Theorem~\ref{t3_2} is equivalent to the following conditions:
\begin{equation}
\label{f3_26}
(h_{m,k,l;n,r,t}, h_{m',k',l';n',r',t'})_{\mathcal{H}} =
(h_{\widetilde m,k,l;n,r,t}, h_{\widetilde m',k',l';n',r',t'})_{\mathcal{H}},
\end{equation}
if $m,m',\widetilde m,\widetilde m',n,n'\in \mathbb{Z}_+$,
$k,l,r,t,k',l',r',t'\in \mathbb{Z}$: $m+m' = \widetilde m + \widetilde m'$;
\begin{equation}
\label{f3_27}
(h_{m,k,l;n,r,t}, h_{m',k',l';n',r',t'})_{\mathcal{H}} =
(h_{m,\widetilde k,l;n,r,t}, h_{m',k',\widetilde l';n',r',t'})_{\mathcal{H}},
\end{equation}
if $m,m',n,n'\in \mathbb{Z}_+$,
$k,\widetilde k, l,r,t,k',l',\widetilde l', r',t'\in \mathbb{Z}$: $k+l' = \widetilde k + \widetilde l'$;
\begin{equation}
\label{f3_28}
(h_{m,k,l;n,r,t}, h_{m',k',l';n',r',t'})_{\mathcal{H}} =
(h_{m,k,\widetilde l;n,r,t}, h_{m',\widetilde k',l';n',r',t'})_{\mathcal{H}},
\end{equation}
if $m,m',n,n'\in \mathbb{Z}_+$,
$k,l,\widetilde l, r,t,k',\widetilde k', l',r',t'\in \mathbb{Z}$: $l+k' = \widetilde l + \widetilde k'$;
\begin{equation}
\label{f3_29}
(h_{m,k,l;n,r,t}, h_{m',k',l';n',r',t'})_{\mathcal{H}} =
(h_{m,k,l;\widetilde n,r,t}, h_{m',k',l';\widetilde n',r',t'})_{\mathcal{H}},
\end{equation}
if $m,m',n,n',\widetilde n,\widetilde n'\in \mathbb{Z}_+$,
$k,l,r,t,k',l',r',t'\in \mathbb{Z}$: $n+n' = \widetilde n + \widetilde n'$;
\begin{equation}
\label{f3_30}
(h_{m,k,l;n,r,t}, h_{m',k',l';n',r',t'})_{\mathcal{H}} =
(h_{m,k,l;n,\widetilde r,t}, h_{m',k',l';n',r',\widetilde t'})_{\mathcal{H}},
\end{equation}
if $m,m',n,n'\in \mathbb{Z}_+$,
$k,l,r,\widetilde r, t,k',l',r',t',\widetilde t' \in \mathbb{Z}$: $r+t' = \widetilde r + \widetilde t'$;
\begin{equation}
\label{f3_31}
(h_{m,k,l;n,r,t}, h_{m',k',l';n',r',t'})_{\mathcal{H}} =
(h_{m,k,l;n,r,\widetilde t}, h_{m',k',l';n',\widetilde r',t'})_{\mathcal{H}},
\end{equation}
if $m,m',n,n'\in \mathbb{Z}_+$,
$k,l,r,t,\widetilde t,k', l',r',\widetilde r',t'\in \mathbb{Z}$: $t+r' = \widetilde t + \widetilde r'$.

As we can see, the solving of the two-dimensional moment problem reduces to a construction of
the set $X(\mathcal{H})$. Let us describe an algorithm for a construction of sequences from
$X(\mathcal{H})$.

Let $\{ g_n \}_{n=1}^\infty$ be an arbitrary orthonormal basis in $\mathcal{H}_0$ obtained
by the Gram-Schmidt orthogonalization procedure from the sequence $\{ h_{m,n} \}_{m,n\in \mathbb{Z}_+}$
indexed by a unique index.

Choose an arbitrary $j\in \mathbb{N}$.
Let $ w(j) = (m,k,l;n,r,t)(j)\in\Omega'$. If we had constructed $\{ h_{m,n} \}_{m,n\in \mathbb{Z}_+}\in X(\mathcal{H})$,
then the two-dimensional moment problem has a solution $\mu$ and
$$ d_j := \| h_{w(j)} \|^2_{\mathcal{H}} = \| x_1^m (x_1+i)^k (x_1-i)^l  x_2^n (x_2+i)^r (x_2-i)^t \|^2_{L^2_\mu} $$
$$ = \int_{\mathbb{R}^2} x_1^{2m} (x_1^2+1)^k (x_1^2+1)^l  x_2^{2n} (x_2^2+1)^r (x_2^2+1)^t d\mu. $$
Therefore $d_j$ are bounded by some constants $M_j=M_j(S)$ depending on the prescribed moments
$S:=\{ s_{m,n} \}_{m,n\in \mathbb{Z}_+}$. (Notice that e.g. $(x_1^2+1)^l \leq 1$, for $l<0$, and for
non-negative $m,k,l;n,r,t$ the values of $d_j$ are determined uniquely).

\noindent
{\bf Step 0. } We set
\begin{equation}
\label{f3_32}
h_{m,0,0;n,0,0} = h_{m,n},\qquad m,n\in \mathbb{Z}_+.
\end{equation}
We check that conditions (\ref{f2_12})-(\ref{f2_15}) (with $h$ instead of $x$) and (\ref{f3_26})-(\ref{f3_31})
are satisfied for $h_{m,0,0;n,0,0}$, $m,n\in \mathbb{Z}_0$. If they are not satisfied, the
two-dimensional moment problem has no solution and we stop the algorithm.

\noindent
{\bf Step 1. } We seek for $h_{w(1)}$ in the following form:
\begin{equation}
\label{f3_33}
h_{w(1)} = \sum_{n=1}^\infty \alpha_{1;n} g_n + \beta_{1;1} e_1,
\end{equation}
with some complex coefficients $\alpha_{1;n},\beta_{1;1}$.

\noindent
Conditions (\ref{f2_12})-(\ref{f2_15}) (with $h$ instead of $x$) and (\ref{f3_26})-(\ref{f3_31})
which include $h_{w(1)}$ and the already constructed $h_{m,k,l;n,r,t}$
are equivalent to a set $L_1$ of linear equations with respect to $\alpha_{1;n}$, $n\in N$,  and
$d_1 = \| h_{w(1)} \|^2_{\mathcal{H}}$. Notice that they depend on $\beta_{1;1}$ only
by $d_1$.
Denote the set of solutions of these equations
by
\begin{equation}
\label{f3_34}
S_1 = \left\{ (\alpha_{1;n}, n\in N; d_1):\ \mbox{ equations from $L_1$ are satisfied} \right\}.
\end{equation}
Set
\begin{equation}
\label{f3_35}
\widehat S_1 = \left\{ (\alpha_{1;n}, n\in N; d_1)\in S_1: \sum_{n=1}^\infty |\alpha_{1;n}|^2 \leq d_1,\
d_1\leq M_1 \right\}.
\end{equation}
Finally, we set
\begin{equation}
\label{f3_36}
G_1 = \left\{ \sum_{n=1}^\infty \alpha_{1;n} g_n +
\left( d_1 - \sum_{n=1}^\infty |\alpha_{1;n}|^2 \right)^{\frac{1}{2}}
e_1:\
(\alpha_{1;n}, n\in N; d_1)\in  \widehat S_1 \right\}.
\end{equation}
The case $G_1 = \emptyset$ is not excluded.

\noindent
{\bf Step r, with $r\geq 2$. } We seek for $h_{w(r)}$ in the following form:
\begin{equation}
\label{f3_37}
h_{w(r)} = \sum_{n=1}^\infty \alpha_{r;n} g_n + \sum_{j=1}^{r} \beta_{r;j} e_j,
\end{equation}
with some complex coefficients $\alpha_{r;n},\beta_{r;j}$.

\noindent
Conditions (\ref{f2_12})-(\ref{f2_15}) (with $h$ instead of $x$) and (\ref{f3_26})-(\ref{f3_31})
which include $h_{w(r)}$ and the already constructed $h_{m,k,l;n,r,t}$
are equivalent to a set $L_r$ of linear equations with respect to $\alpha_{r;n}$, $n\in N$,
$\beta_{r;j}$, $1\leq j\leq r-1$,
and
$d_r = \| h_{w(r)} \|^2_{\mathcal{H}}$, and
{\it depending on parameters $( h_{w(1)},h_{w(2)},...,h_{w(r-1)} )\in G_{r-1}$ }.

\noindent
Notice that these linear equations depend on $\beta_{r;j}$ only
by $d_r$.
Denote the set of solutions of these equations
by
$$ S_r = \left\{ ( \alpha_{r;n}, n\in N; \beta_{r;j}, 1\leq j\leq r-1; d_r; h_{w(1)},h_{w(2)},...,h_{w(r-1)}  ): \right. $$
$$ \left. ( h_{w(1)},h_{w(2)},...,h_{w(r-1)} )\in G_{r-1}, \right. $$
\begin{equation}
\label{f3_38}
\left. \mbox{ and equations from $L_r$ with parameters $( h_{w(1)},h_{w(2)},...,h_{w(r-1)} )$, are satisfied} \right\}.
\end{equation}
Set
$$ \widehat S_r = \left\{ (\alpha_{r;n}, n\in N; \beta_{r;j}, 1\leq j\leq r-1; d_r;
h_{w(1)},h_{w(2)},...,h_{w(r-1)} )\in S_r: \right. $$
\begin{equation}
\label{f3_39}
\left. \sum_{n=1}^\infty |\alpha_{r;n}|^2 + \sum_{j=1}^{r-1} |\beta_{r;j}|^2 \leq d_r,\ d_r\leq M_r \right\}.
\end{equation}
Finally, we set
$$ G_r = \left\{ \left(
h_{w(1)},h_{w(2)},...,h_{w(r-1)}, \right. \right. $$
$$ \left. \left.
\sum_{n=1}^\infty \alpha_{r;n} g_n + \sum_{j=1}^{r-1} \beta_{r;j} e_j +
\left( d_r - \sum_{n=1}^\infty |\alpha_{r;n}|^2 - \sum_{j=1}^{r-1} |\beta_{r;j}|^2
\right)^{\frac{1}{2}} e_r
\right): \right. $$
\begin{equation}
\label{f3_40}
\left. (\alpha_{r;n}, n\in N; \beta_{r;j}, 1\leq j\leq r-1; d_r;
h_{w(1)},h_{w(2)},...,h_{w(r-1)} ) \in  \widehat S_r \right\}.
\end{equation}
(The case $G_r = \emptyset$ is not excluded.)

\noindent
{\bf Final step. } Consider a space $\mathbf{H}$ of sequences
\begin{equation}
\label{f3_41}
\mathbf{h} = (h_1,h_2,h_3,...),\qquad h_r\in \mathcal{H},\ r\in \mathbb{N},
\end{equation}
with the norm given by
\begin{equation}
\label{f3_41_1}
\| \mathbf{h} \|_{\mathbf{H}} = \sup_{r\in \mathbb{N}} \frac{1}{ \sqrt{M_r} } \| h_r \|_{\mathcal{H}} <\infty.
\end{equation}
For arbitrary $(h_1,...,h_{r})\in G_r$, we put into correspondence elements $\mathbf{h}\in \mathbf{H}$
of the following form
\begin{equation}
\label{f3_42}
\mathbf{h} = (h_1,...,h_{r},g_{r+1},g_{r+2},...):\ g_j\in \mathcal{H},\ \| g_j \|_{\mathcal{H}} \leq
\sqrt{M_j},\ j>r.
\end{equation}
Thus, the set $G_r$ is mapped onto a set $\mathbf{G}_r\subset \mathbf{H}$.
If $G_r = \emptyset$, we set $\mathbf{G}_r = \emptyset$.
Observe that all elements of $\mathbf{G}_r$ has the norm less or equal to $1$.
Set
\begin{equation}
\label{f3_43}
\mathbf{G} = \bigcap_{r=1}^\infty \mathbf{G}_r.
\end{equation}
If $\mathbf{G}\not= \emptyset$, then
to each $(g_1,g_2,...)\in \mathbf{G}$, we put into correspondence a sequence
$\mathfrak{H} = \{ h_{m,k,l;n,r,t} \}_{(m,k,l;n,r,t)\in\Omega}$ such that~(\ref{f3_32}) holds
and
\begin{equation}
\label{f3_44}
h_{w(r)} := g_r,\qquad r\in \mathbb{N}.
\end{equation}
We state that $\mathfrak{H}\in X(\mathcal{H})$.
In fact, conditions (\ref{f2_12})-(\ref{f2_15}) (with $h$ instead of $x$) and (\ref{f3_26})-(\ref{f3_31})
are satisfied for $h_{m,0,0;n,0,0}$, $m,n\in \mathbb{Z}_+$, by Step~0.
If one of these equations include $h_{w(r)}$ with $r\geq 1$, then we choose the maximal appearing
index $r$. Since $( h_{w(1)},...,h_{w(r)} )\in G_{r}$, then this equation is satisfied.
Condition~2) is satisfied by the construction.

\noindent
Thus, if $\mathbf{G}\not= \emptyset$, then using $\mathfrak{H}$ we can construct a solution
of the two-dimensional moment problem in the described
above manner.

\begin{theorem}
\label{t3_3}
Let the two-dimensional moment problem~(\ref{f1_1}) be given and condition~(\ref{f3_2}) holds.
Choose an arbitrary model
space $\mathcal{H}$ with a sequence $\{ h_{m,n} \}_{m,n\in \mathbb{Z}_+}$,
satisfying~(\ref{f3_3}) and fix it. The moment problem has a solution
if and only if
conditions (\ref{f2_12})-(\ref{f2_15}) (with $h$ instead of $x$) and (\ref{f3_26})-(\ref{f3_31})
are satisfied for $h_{m,0,0;n,0,0}:= h_{m,n}$, $m,n\in \mathbb{Z}_+$, and
\begin{equation}
\label{f3_45}
\mathbf{G} \not= \emptyset,
\end{equation}
where $\mathbf{G}$ is constructed by~(\ref{f3_43}) according to the algorithm.

If the latter conditions are satisfied then
to each $(g_1,g_2,...)\in \mathbf{G}$, we put into correspondence a sequence
$\mathfrak{H} = \{ h_{m,k,l;n,r,t} \}_{(m,k,l;n,r,t)\in\Omega}$ such that~(\ref{f3_32}) holds
and
\begin{equation}
\label{f3_45_1}
h_{w(r)} := g_r,\qquad r\in \mathbb{N}.
\end{equation}
This sequence belongs to $X(\mathcal{H})$ and
the unique solution of the extended two-dimensional moment problem with moments~(\ref{f3_25})
gives a solution $\mu$ of the two-dimensional
moment problem.
Moreover, all solutions of the two-dimensional moment problem can be obtained in this way.
\end{theorem}
\begin{proof}
The sufficiency of the conditions in the statement of the Theorem for the solvability of
the two-dimensional moment problem was shown before the statement of the Theorem.
Let us show that these conditions are necessary.

\noindent
Let $\mu$ be a solution of the two-dimensional moment problem. By Theorem~\ref{t3_2}
the set $X(\mathcal{H})$ is not empty. Choose an arbitrary
$\widehat{\mathfrak{H}} = \{ \widehat h_{m,k,l;n,r,t} \}_{(m,k,l;n,r,t)\in\Omega} \in X(\mathcal{H})$.
By conditions~3),4) we see that
conditions (\ref{f2_12})-(\ref{f2_15}) (with $\widehat h$ instead of $x$) and (\ref{f3_26})-(\ref{f3_31})
are satisfied. In  particular, they are satisfied for $\widehat h_{m,0,0;n,0,0} = h_{m,n}$, $m,n\in \mathbb{Z}_+$.

\noindent
Comparing condition~2) with Steps 1 and $r$ for $r\geq 2$, we see that
\begin{equation}
\label{f3_46}
(\widehat h_{w(1)},...,\widehat h_{w(r)} ) \in G_r,\qquad r\in \mathbb{N}.
\end{equation}
Therefore elements
\begin{equation}
\label{f3_47}
(\widehat h_{w(1)},...,\widehat h_{w(r)}, g_{r+1}, g_{r+2},... ) \in \mathbf{G}_r,\qquad r\in \mathbb{N},
\end{equation}
where $g_j\in \mathcal{H}:\ \| g_j \|_{\mathcal{H}}\leq \sqrt{M_j}$, for $j > r$ are arbitrary.
Thus, the element
\begin{equation}
\label{f3_48}
\widehat{\mathbf{h}}:=(\widehat h_{w(1)}, \widehat h_{w(2)}, \widehat h_{w(3)},... ) \in
\mathbf{G}_r,\qquad r\in \mathbb{N}.
\end{equation}
and
\begin{equation}
\label{f3_49}
\widehat{\mathbf{h}}\in \bigcap_{r\in \mathbb{N}} \mathbf{G}_r = \mathbf{G}.
\end{equation}
Therefore $\mathbf{G} \not= \emptyset$.

If the conditions of the Theorem are satisfied then
to each $\mathbf{g} = (g_1,g_2,...)\in \mathbf{G}$, we put into correspondence a sequence
$\mathfrak{H} = \mathfrak{H}(\mathbf{g}) = \{ h_{m,k,l;n,r,t} \}_{(m,k,l;n,r,t)\in\Omega}$ such
that~(\ref{f3_32}) and~(\ref{f3_45_1}) hold.
Then $\mathfrak{H}\in X(\mathcal{H})$, as it was shown before the statement of the Theorem.
The sequence $\mathfrak{H}$ generates a solution of the extended two-dimensional moment problem and
of the two-dimensional moment problem, see considerations after the proof of Theorem~\ref{t3_2}.

\noindent
It remains to show that all solutions of the two-dimensional moment problem can be obtained
in this way.
Since elements of $X(\mathcal{H})$ generate all solutions of the two-dimensional moment problem
(see considerations after the proof of Theorem~\ref{t3_2}), it remains to prove that
\begin{equation}
\label{f3_50}
\{ \mathfrak{H}(\mathbf{g}):\ \mathbf{g}\in \mathbf{G} \} = X(\mathcal{H}).
\end{equation}
Denote the set on the left-hand side by $X_1$. It was shown that $X_1\subseteq X(\mathcal{H})$.
On the other hand, choose an arbitrary
$\widetilde{\mathfrak{H}} = \{ \widetilde h_{m,k,l;n,r,t} \}_{(m,k,l;n,r,t)\in\Omega} \in X(\mathcal{H})$.
Repeating the construction at the beginning of this proof we obtain that
\begin{equation}
\label{f3_51}
\widetilde{\mathbf{h}} :=
(\widetilde h_{w(1)}, \widetilde h_{w(2)}, \widetilde h_{w(3)},... ) \in \mathbf{G}.
\end{equation}
Observe that
\begin{equation}
\label{f3_52}
\mathfrak{H}(\widetilde{\mathbf{h}}) =\widetilde{\mathfrak{H}}.
\end{equation}
Therefore $X(\mathcal{H})\subseteq X_1$ and relation~(\ref{f3_50}) holds.
\end{proof}

\begin{remark}
The truncated two-dimensional moment problem can be considered in a
similar manner. Moreover, the set of indices of the known elements $h_{m,n}$ will be finite in this case
and therefore equations in the $r$-th step of the algorithm will form finite systems
of linear equations. Thus, the $r$-th step could be easily performed using computer.
\end{remark}

\begin{remark}
Consider the following system of $r$ linear equations:
\begin{equation}
\label{f3_53}
A^{1}
\left(
\begin{array}{ccc} x_1\\
x_2\\
\vdots \end{array}
\right)
=
\left(
\begin{array}{cccc} f_1^1\\
f_2^1\\
\vdots\\
f_r^1 \end{array}
\right),
\end{equation}
where $A^{1} = (a^1_{i,j})_{1\leq i\leq r; j\in \mathbb{N}}$ is a given complex numerical matrix,
$f_i^1$, $1\leq i\leq r$, are given complex numbers, and $x_j$, $j\in \mathbb{N}$, are
unknown complex numbers; $r\in \mathbb{N}$.

The Gauss algorithm allows to solve this system explicitly. Let us briefly describe this.

\noindent
{\it Step 1. } If $A^1 = 0$ then the algorithm stops. Conditions of solvability and the set of solutions
are obvious in this case.

\noindent
If $A^1 \not= 0$, let $m_1$-th column of $A^1$ be the first non-zero column of $A^1$.
Interchanging equations we set the non-zero element of this column in the first row
and divide this equation by this element.
Then we exclude $x_{m_1}$ from the  rest of equations. We get the following system:
\begin{equation}
\label{f3_54}
x_{m_1} + a^2_{1,m_1+1} x_{m_1+1} + a^2_{1,m_1+2} x_{m_1+2} + ... = f_1^2,
\end{equation}
\begin{equation}
\label{f3_55}
A^{2}
\left(
\begin{array}{ccc} x_{m_1+1}\\
x_{m_1+2}\\
\vdots \end{array}
\right)
=
\left(
\begin{array}{cccc} f_2^2\\
f_3^2\\
\vdots\\
f_r^2 \end{array}
\right),
\end{equation}
where $A^{2}$ is a given complex numerical matrix with $r-1$ rows,
$f_i^2$, $1\leq i\leq r$, and $a^2_{1,j}$, $j\geq m_1+1$, are
given complex numbers.

Then we repeat the same construction for the linear system~(\ref{f3_55}). After a finite number of
steps the algorithm stops.
Then we exclude $x_{m_t}, x_{m_{t-1}},...,x_{m_1}$ from the previous equations ($t\leq r$).

\noindent
Thus, the numbers $x_j$: $j\not= m_1,m_2,...,m_t$ can be chosen arbitrary such that the corresponding
series in~(\ref{f3_53}) converge, and $x_{m_1},x_{m_2},...,x_{m_t}$ are defined uniquely.
If $t<r$, we additionally have the solvability conditions which follow from~(\ref{f3_55}) in the last step.

Observe that if  we have an infinite number of equations in~(\ref{f3_53}), we can choose an increasing number
of equations and then construct the intersection of the solution sets.
\end{remark}

\noindent
{\bf A modified algorithm.}
Notice that in~Theorem~\ref{t3_3}
the correspondence between the parameters set $\mathbf{G}$ and solutions of the two-dimensional
moment problem is not necessarily bijective.
The algorithm may be modified to make this correspondence one-to-one. The following modified algorithm is more
complicated. If we only need to check the solvability or the bijection is not necessary for our purposes,
we can use the original algorithm.

\noindent
First, condition~2) of Theorem~\ref{t3_2} may be replaced by the following more precise condition:

\noindent
2) Set $\Lambda := \{ r\in \mathbb{N}:\  (h_{w(r)},e_r) > 0 \}$. Then
\begin{equation}
\label{f3_55_1}
h_{w(r)} \in \mathcal{H}_0 \oplus \left(
\bigoplus_{j\in \Lambda:\ 1\leq j\leq r} \mathcal{H}_j
\right), \mbox{and } (h_{w(r)},e_r) \geq 0,\ r\in \mathbb{N}.
\end{equation}

\noindent
The necessity of this condition was shown before Theorem~\ref{t3_2}, while condition~2) was not used
in the proof of the sufficiency of Theorem~\ref{t3_2}.

\noindent
As before, we denote a set of sequences
$\{ h_{m,k,l;n,r,t} \}_{(m,k,l;n,r,t)\in\Omega}$, in $\mathcal{H}$ satisfying
conditions 1)-4) of Theorem~\ref{t3_2} by $X=X(\mathcal{H})$. Observe that the modified $X$ is
a subset of the original one.
The same arguments show that the new $X(\mathcal{H})$ generates all solutions of the two-dimensional
moment problem, as well.

\noindent
Step~0 and Step~1 of the algorithm will be the same as before.

\noindent
We set
\begin{equation}
\label{f3_56}
\mathcal{H}^k := \mathcal{H}_0 \oplus \left(
\bigoplus_{1\leq j\leq k} \mathcal{H}_j
\right),\quad \mathcal{H}^k_+ := \{ h\in \mathcal{H}^k:\ (h,e_k) > 0 \},\qquad k\in \mathbb{N}.
\end{equation}
In the $r$-th step we shall proceed in the following way ($r\geq 2$).

\noindent
Choose an arbitrary $(h_{w(1)},h_{w(2)},...,h_{w(r-1)})\in G_{r-1}$.
Observe that by the construction in the $(r-1)$-th step we have
\begin{equation}
\label{f3_57}
h_{w(j)} \in \mathcal{H}^{k-1}  \mbox{ or } h_{w(j)} \in \mathcal{H}^k_+,\qquad 1\leq k\leq r-1.
\end{equation}
Set
\begin{equation}
\label{f3_58}
S_r := \{ \vec s = (s_1,s_2,...,s_{r-1}):\ s_j=1 \mbox{ or } s_j=0,\ 1\leq j\leq r-1 \};
\end{equation}
and
$$ \mathbf{H}^r_{\vec s} := \{ (h_1,h_2,...,h_{r-1}):\ h_j\in \mathcal{H}^{j-1} \mbox{ if } s_j=0;\
h_j\in \mathcal{H}^{j}_+ \mbox{ if } s_j=1;\ 1\leq j\leq r-1 \}, $$
\begin{equation}
\label{f3_59}
\vec{s}\in S^r.
\end{equation}
Observe that $S_r$ is a finite set of $2^{r-1}$ binary numbers. By~(\ref{f3_57}) we obtain that
\begin{equation}
\label{f3_60}
(h_{w(1)},h_{w(2)},...,h_{w(r-1)})\in G_{r-1}\cap \mathbf{H}^r_{\vec s},\quad \mbox{for some }
\vec{s}\in S^r.
\end{equation}
Set
\begin{equation}
\label{f3_61}
\Gamma_{r-1,\vec{s}} := G_{r-1}\cap \mathbf{H}^r_{\vec s},\quad \vec{s}\in S^r.
\end{equation}
Notice that
\begin{equation}
\label{f3_62}
\Gamma_{r-1,\vec{s}_1} \cap \Gamma_{r-1,\vec{s}_2} = \emptyset,\quad \vec{s}_1,\vec{s}_2\in S^r,
\end{equation}
and
\begin{equation}
\label{f3_63}
G_{r-1} = \bigcup_{\vec{s}\in S^r} \Gamma_{r-1,\vec{s}}.
\end{equation}
Choose an arbitrary $\vec{s}\in S^r$.
We seek for $h_{w(r)}$ in the following form:
\begin{equation}
\label{f3_64}
h_{w(r)} = \sum_{n=1}^\infty \alpha_{r;n} g_n + \sum_{1\leq j\leq r-1:\ s_j = 1} \beta_{r;j} e_j +
\beta_{r;r} e_r,
\end{equation}
with some complex coefficients $\alpha_{r;n},\beta_{r;j}$: $\beta_{r;r}\geq 0$.

\noindent
Conditions (\ref{f2_12})-(\ref{f2_15}) (with $h$ instead of $x$) and (\ref{f3_26})-(\ref{f3_31})
which include $h_{w(r)}$ and the already constructed $h_{m,k,l;n,r,t}$
are equivalent to a set $L_r(\vec{s})$ of linear equations with respect to $\alpha_{r;n}$, $n\in \mathbb{N}$,
$\beta_{r;j}$, $1\leq j\leq r-1:\ s_j=1$,
and
$d_r = \| h_{w(r)} \|^2_{\mathcal{H}}$, and
{\it depending on parameters $( h_{w(1)},h_{w(2)},...,h_{w(r-1)} )\in \Gamma_{r-1,\vec{s}}$ }.

\noindent
Notice that these linear equations depend on $\beta_{r;j}$ only
by $d_r$.
Denote the set of solutions of these equations
by
$$ S_r(\vec{s})
= \left\{ ( \alpha_{r;n}, n\in \mathbb{N}; \beta_{r;j}, 1\leq j\leq r-1:\ s_j=1; d_r;
h_{w(1)},h_{w(2)},...,h_{w(r-1)}  ): \right. $$
$$ \left. ( h_{w(1)},h_{w(2)},...,h_{w(r-1)} )\in \Gamma_{r-1;\vec{s}}, \right. $$
\begin{equation}
\label{f3_65}
\left. \mbox{ and equations from $L_r(\vec{s})$ with parameters
$( h_{w(1)},h_{w(2)},...,h_{w(r-1)} )$, are satisfied} \right\}.
\end{equation}
Set
$$ \widehat S_r(\vec{s}) = \left\{ (\alpha_{r;n}, n\in N; \beta_{r;j}, 1\leq j\leq r-1:\ s_j=1; d_r;
h_{w(1)},h_{w(2)},...,h_{w(r-1)} )\in S_r(\vec{s}): \right. $$
\begin{equation}
\label{f3_66}
\left. \sum_{n=1}^\infty |\alpha_{r;n}|^2 + \sum_{1\leq j\leq r-1:\ s_j=1} |\beta_{r;j}|^2 \leq d_r,\
d_r\leq M_r \right\}.
\end{equation}
Finally, we set
$$ G_r(\vec{s}) = \left\{ \left(
h_{w(1)},h_{w(2)},...,h_{w(r-1)}, \right. \right. $$
$$ \left. \left.
\sum_{n=1}^\infty \alpha_{r;n} g_n + \sum_{1\leq j\leq r-1:\ s_j=1} \beta_{r;j} e_j +
\left( d_r - \sum_{n=1}^\infty |\alpha_{r;n}|^2 - \sum_{1\leq j\leq r-1:\ s_j=1} |\beta_{r;j}|^2
\right)^{\frac{1}{2}} e_r
\right): \right. $$
\begin{equation}
\label{f3_67}
\left. (\alpha_{r;n}, n\in \mathbb{N}; \beta_{r;j}, 1\leq j\leq r-1:\ s_j=1; d_r;
h_{w(1)},h_{w(2)},...,h_{w(r-1)} ) \in  \widehat S_r(\vec{s}) \right\}.
\end{equation}
The case $G_r(\vec{s}) = \emptyset$ is not excluded.

\noindent
We set
\begin{equation}
\label{f3_68}
G_r := \bigcup_{\vec{s}\in S^r} G_r(\vec{s}).
\end{equation}
The final step is the same as for the original algorithm. Thus, we obtain the set $\mathbf{G}$.

\noindent
To each $\mathbf{g} =(g_1,g_2,...)\in \mathbf{G}$, we put into correspondence a sequence
$\mathfrak{H} = \mathfrak{H}(\mathbf{g})
= \{ h_{m,k,l;n,r,t} \}_{(m,k,l;n,r,t)\in\Omega}$ such that~(\ref{f3_32}),(\ref{f3_44}) hold.
We state that $\mathfrak{H}(\mathbf{g})\in X(\mathcal{H})$.
Observe that the  set $G_r$ in~(\ref{f3_68}) is a subset of $G_r$ for the original algorithm.
Therefore the set $\mathbf{G}$ is a subset of $\mathbf{G}$ for the original algorithm.
Thus, $\mathfrak{H}(\mathbf{g})$ belongs to the old $X(\mathcal{H})$. To show that it
belongs to the modified $X(\mathcal{H})$ it remains to verify~(\ref{f3_55_1}).
Observe that $( h_{w(1)}, h_{w(2)},...,h_{w(r)})\in G_{r}(\vec{s})$, for some $\vec{s}\in S^r$.
The condition $(h_{w(r)},e_r)\geq 0$ follows from the construction of $h_{w(r)}$ in the $r$-th step.
Set
\begin{equation}
\label{f3_69}
\Lambda_0(r) := \left\{ \begin{array}{cc} \{ 1\leq j\leq r-1:\ s_j=1 \} \cup {r}, & \mbox{if }(h_{w(r)},e_r) > 0 \\
\{ 1\leq j\leq r-1:\ s_j=1 \}, & \mbox{if }(h_{w(r)},e_r) = 0\end{array}\right..
\end{equation}
By~(\ref{f3_64}) we get
\begin{equation}
\label{f3_70}
h_{w(r)} \in \mathcal{H}_0 \oplus \left(
\bigoplus_{j\in \Lambda_0(r)} \mathcal{H}_j
\right).
\end{equation}
Thus, it remains to verify that $\Lambda_0(r) = \{j\in \Lambda:\ 1\leq j\leq r \} =: \Lambda(r)$.
But for $1\leq j\leq r-1$, conditions
$s_j=1$ and $(h_{w(j)},e_j)>0$ are equivalent.
Consequently, we obtain $\mathfrak{H}(\mathbf{g}) \in X(\mathcal{H})$.

Theorem~\ref{t3_3} remains true if we replace words "according to the algorithm" by the
words "according to the modified algorithm", and add the following sentence:
"The correspondence between elements of $\mathbf{G}$ and solutions of the two-dimensional
moment problem is bijective".
Let us check this last assertion (and the rest of the proof is similar).

\noindent
The correspondence between $\mathbf{G}$ and $X(\mathcal{H})$ is obviously bijective.
Let
$\mathfrak{H}_j = \{ h^j_{m,k,l;n,r,t} \}_{(m,k,l;n,r,t)\in\Omega} \in X(\mathcal{H})$,
$j=1,2$, be different: $\mathfrak{H}_1\not = \mathfrak{H}_2$. They
produce solutions $\mu_1$ and $\mu_2$ of the two-dimensional moment problem, respectively.
Suppose that $\mu_1 = \mu_2 = \mu$. Recall that $\mu_j$ is constructed as a solution
of the corresponding extended two-dimensional moment problem with moments
$u_{m,k,l;n,r,t}^j = \varphi_j(m,k,l;n,r,t)$, $j=1,2$ (see the proof of Theorem~\ref{t3_2}).
Here $\varphi_j(m,k,l;n,r,t)$ is from Condition~4) for $\mathfrak{H}_j$, $j=1,2$.
Therefore
$\varphi_1(m,k,l;n,r,t) = \varphi_2(m,k,l;n,r,t)$.
By condition~4) of Theorem~\ref{t3_2} this means that
\begin{equation}
\label{f3_71}
(h^1_{m,k,l;n,r,t},h^1_{m',k',l';n',r',t'})_{\mathcal{H}} =
(h^2_{m,k,l;n,r,t},h^2_{m',k',l';n',r',t'})_{\mathcal{H}},
\end{equation}
for all $(m,k,l;n,r,t),(m',k',l';n',r',t')\in \Omega$.

\noindent
Choose the minimal $r$, $r\in \mathbb{N}$, such that
\begin{equation}
\label{f3_72}
h^1_{w(r)} \not= h^2_{w(r)}.
\end{equation}
By~(\ref{f3_32}),(\ref{f3_71}) we obtain
\begin{equation}
\label{f3_73}
P^{\mathcal{H}}_{\mathcal{H}_0} h^1_{w(r)} = P^{\mathcal{H}}_{\mathcal{H}_0} h^2_{w(r)} =: h_0.
\end{equation}
By condition~(\ref{f3_55_1}) we may write
$$ h^1_{w(r)} = h_0 + \sum_{1\leq j\leq r-1:\ (h^1_{w(j)},e_j) > 0} \gamma_{r;j}^1 e_j +
\gamma_{r;r}^1 e_r,\quad  \gamma_{r;j}^1\in \mathbb{C},\ \gamma_{r;r}^1\geq 0; $$
\begin{equation}
\label{f3_74}
h^2_{w(r)} = h_0 + \sum_{1\leq j\leq r-1:\ (h^2_{w(j)},e_j) > 0} \gamma_{r;j}^2 e_j +
\gamma_{r;r}^2 e_r,\quad  \gamma_{r;j}^2\in \mathbb{C},\ \gamma_{r;r}^2\geq 0.
\end{equation}
Since $h^1_{w(j)} = h^2_{w(j)} =: h_{w(j)}$, $1\leq j\leq r-1$, we get
\begin{equation}
\label{f3_75}
\{j:\ 1\leq j\leq r-1,\ (h^1_{w(j)},e_j) > 0 \} = \{j:\ 1\leq j\leq r-1,\ (h^2_{w(j)},e_j) > 0 \} =:\widehat\Lambda.
\end{equation}
Therefore
\begin{equation}
\label{f3_76}
h^a_{w(r)} = h_0 + \sum_{j\in \widehat\Lambda} \gamma_{r;j}^a e_j +
\gamma_{r;r}^a e_r,\quad  \gamma_{r;j}^a\in \mathbb{C},\ \gamma_{r;r}^a\geq 0,\ a=1,2.
\end{equation}
Suppose that there exists $j\in\widehat\Lambda$ such that $\gamma_{r;j}^1\not= \gamma_{r;j}^2$.
Let $j_0$ be the minimal such index $j$.
Since $j_0\in\widehat\Lambda$, we get
$$ \zeta_{j_0} := (h_{w(j_0)}, e_{j_0}) = (h^a_{w(j_0)}, e_{j_0}) > 0,\qquad a=1,2, $$
and
\begin{equation}
\label{f3_77}
h_{w(j_0)} = \zeta_{j_0} e_{j_0} + u_{j_0-1},\quad u_{j_0-1}\in \mathcal{H}^{j_0-1}.
\end{equation}
Then
\begin{equation}
\label{f3_78}
e_{j_0} = \frac{1}{\zeta_{j_0}} ( h_{w(j_0)} - u_{j_0-1} );
\end{equation}
and
$$ \gamma_{r;j_0}^a = (h^a_{w(r)}, e_{j_0}) = \frac{1}{ \overline{\zeta_{j_0}} } (h^a_{w(r)},
h_{w(j_0)} - u_{j_0-1}) $$
$$ = \frac{1}{ \overline{\zeta_{j_0}} } (h^a_{w(r)},
h_{w(j_0)}) - \frac{1}{ \overline{\zeta_{j_0}} } (h^a_{w(r)},u_{j_0-1}) $$
\begin{equation}
\label{f3_79}
= \frac{1}{ \overline{\zeta_{j_0}} } (h^a_{w(r)},
h_{w(j_0)}) - \frac{1}{ \overline{\zeta_{j_0}} } ( h_0 + \sum_{j\in \widehat\Lambda: j<j_0} \gamma_{r;j}^a e_j,
u_{j_0-1}).
\end{equation}
By~(\ref{f3_71}) and our assumption about $j_0$ we obtain $\gamma_{r;j_0}^1 = \gamma_{r;j_0}^2$.
This contradiction means that
$\gamma_{r;j}^1 = \gamma_{r;j}^2$, $\forall j\in\widehat\Lambda$. Therefore
\begin{equation}
\label{f3_80}
h^a_{w(r)} = \widehat h + \gamma_{r;r}^a e_r,\quad  \widehat h\in \mathcal{H}^{r-1},\
\gamma_{r;r}^a\geq 0,\ a=1,2.
\end{equation}
\noindent
Observe that
$$ \| h_{w(r)}^a \|^2 = \| \widehat h \|^2 + | \gamma_{r;r}^a |^2,\quad a=1,2. $$
By~(\ref{f3_71}) we conclude that $\gamma_{r;r}^1 = \gamma_{r;r}^2$.
Therefore $h_{w(r)}^1 = h_{w(r)}^2$.
We obtained a contradiction with~(\ref{f3_72}).
The proof of the last assertion for the modified Theorem~\ref{t3_3} is complete.

\section{On a connection with the complex moment problem.}
In this section we shall analyze the complex moment problem:
to find a non-negative Borel measure $\sigma$ in the complex plane such that
\begin{equation}
\label{f4_1}
\int_\mathbb{C} z^m \overline{z}^n d\sigma = a_{m,n},\qquad m,n\in \mathbb{Z}_+,
\end{equation}
where $\{ a_{m,n} \}_{m,n\in \mathbb{Z}_+}$ is a prescribed sequence of complex numbers.

Recall the canonical identification
of $\mathbb{C}$ with $\mathbb{R}^2$:
\begin{equation}
\label{f4_2}
z = x_1 + x_2 i,\ x_1 = \mathop{\rm Re}\nolimits z,\ x_2 = \mathop{\rm Im}\nolimits z,\quad
z\in \mathbb{C},\ (x_1,x_2)\in \mathbb{R}^2.
\end{equation}
Let $\sigma$ be a solution of the complex moment problem~(\ref{f4_1}).
The measure $\sigma$, viewed as a measure in $\mathbb{R}^2$, we shall denote by $\mu_\sigma$. Then
$$ s_{m,n} := \int_{\mathbb{R}^2} x_1^m x_2^n d\mu_\sigma = \int_\mathbb{C} \left( \frac{z+\overline{z}}{2} \right)^m
\left( \frac{z-\overline{z}}{2i} \right)^n d\sigma $$
$$ = \frac{1}{2^m (2i)^{n}} \sum_{k=0}^m \sum_{j=0}^n
C^m_k C^n_j (-1)^{n-j}\int_\mathbb{C} z^{k+j} \overline{z}^{m-k+n-j} d\sigma $$
\begin{equation}
\label{f4_3}
= \frac{1}{2^m (2i)^{n}} \sum_{k=0}^m \sum_{j=0}^n (-1)^{n-j} C^m_k C^n_j a_{k+j,m-k+n-j},
\end{equation}
where $C^n_k = \frac{n!}{k!(n-k)!}$.
Then
$$ a_{m,n} = \int_\mathbb{C} z^m \overline{z}^n d\sigma = \int_{\mathbb{R}^2} (x_1+ix_2)^m
(x_1-ix_2)^n d\mu_\sigma  $$
$$ = \sum_{r=0}^m \sum_{l=0}^n C^m_r C^n_l (-1)^{n-l} \int_{\mathbb{R}^2} x_1^{r+l} (ix_2)^{m-r+n-l} d\mu_\sigma $$
$$ = \sum_{r=0}^m \sum_{l=0}^n C^m_r C^n_l (-1)^{n-l} i^{m-r+n-l}
s_{r+l,m-r+n-l}; $$
and therefore
\begin{equation}
\label{f4_4}
a_{m,n} = \sum_{r=0}^m \sum_{l=0}^n C^m_r C^n_l (-1)^{n-l} i^{m-r+n-l}
s_{r+l,m-r+n-l},\quad m,n\in \mathbb{Z}_+,
\end{equation}
where
\begin{equation}
\label{f4_5}
s_{m,n}
= \frac{1}{2^m (2i)^{n}} \sum_{k=0}^m \sum_{j=0}^n (-1)^{n-j} C^m_k C^n_j a_{k+j,m-k+n-j},\quad m,n\in \mathbb{Z}_+.
\end{equation}
Since $\mu_\sigma$ is a solution of the two-dimensional moment problem, then
conditions of Theorem~\ref{t3_3} hold.
\begin{theorem}
\label{t4_1}
Let the complex moment problem~(\ref{f4_1}) be given.
This problem has a solution if an only if conditions of Theorem~\ref{t3_3} and~(\ref{f4_4})
with $s_{m,n}$ defined by~(\ref{f4_5})  hold.
\end{theorem}
\begin{proof}
It remains to prove the sufficiency. Suppose that for the complex moment problem~(\ref{f4_1})
conditions of Theorem~\ref{t3_3} and~(\ref{f4_4}) with $s_{m,n}$ defined by~(\ref{f4_5}) hold.
By Theorem~\ref{t3_3} we obtain that
there exists a solution $\mu$ of the two-dimensional moment problem with moments $s_{m,n}$.

The measure $\mu$, viewed as a measure in $\mathbb{C}$, we shall denote by $\sigma_\mu$. Then
$$ \int_\mathbb{C} z^m \overline{z}^n d\sigma_\mu = \int_{\mathbb{R}^2} (x_1+ix_2)^m (x_1-ix_2)^n d\mu  $$
$$ = \sum_{r=0}^m \sum_{l=0}^n C^m_r C^n_l (-1)^{n-l} \int_{\mathbb{R}^2} x_1^{r+l} (ix_2)^{m-r+n-l} d\mu $$
$$ = \sum_{r=0}^m \sum_{l=0}^n C^m_r C^n_l (-1)^{n-l} i^{m-r+n-l}
s_{r+l,m-r+n-l} = a_{m,n}, $$
where the last equality follows from~(\ref{f4_4}).
\end{proof}

\begin{theorem}
\label{t4_2}
Let the complex moment problem~(\ref{f4_1}) be given and
conditions of Theorem~\ref{t3_3} and~(\ref{f4_4})
with $s_{m,n}$ defined by~(\ref{f4_5})  hold.
Let $\Psi$ be a set of all solutions of the complex moment problem~(\ref{f4_1})
and $\Phi$ be a set of all solutions of the two-dimensional moment problem~(\ref{f1_1})
with $s_{m,n}$ defined by~(\ref{f4_5}).
Then
\begin{equation}
\label{f4_6}
\Psi = \{ \sigma_\mu:\ \mu\in \Phi\}.
\end{equation}
Therefore all solutions of the moment problem~(\ref{f4_1}) are described
by Theorem~\ref{t3_3}.
\end{theorem}
\begin{proof}
The proof is straightforward.
\end{proof}
Of course, this Theorem holds for the modified version of Theorem~\ref{t3_3}, as well.

\bibliographystyle{amsplain}

\end{document}